\documentclass[a4paper, 12pt]{article}

\usepackage[dvipsnames]{xcolor}
\usepackage{amsmath, amsthm, mathtools}
\usepackage{amsfonts, mathrsfs}
\usepackage{amssymb}
\usepackage{paralist}
\usepackage{extarrows}
\usepackage{multirow}
\usepackage[colorlinks,linkcolor=blue, citecolor=Green]{hyperref}
\usepackage[numbers, sort]{natbib}

\usepackage[margin=1in]{geometry}

\usepackage{titlesec}
\titleformat{\subsection}{\normalfont\large\bfseries\itshape}{\thesubsection}{1em}{}

\usepackage{tikz}
\usetikzlibrary{calc}

\usepackage{caption, makecell, array}

\newcommand{\R}{\mathbb{R}}
\renewcommand{\S}{\mathbb{S}}

\numberwithin{equation}{section}
\newtheorem{theorem}{Theorem}[section]
\newtheorem*{theorem*}{Theorem}
\newtheorem{lemma}{Lemma}[section]
\newtheorem{proposition}{Proposition}[section]
\newtheorem{remark}{Remark}[section]
\newtheorem{corollary}{Corollary}[section]
\newtheorem{alphatheorem}{Theorem}[section]

\title{\bf Bivariate Hardy-Sobolev Inequality and Its Sharp Stability}

\author{
    Yingfang Zhang\thanks{Department of  Mathematical Sciences, Tsinghua University; E-mail: zhangyf22@mails.tsinghua.edu.cn},
    \quad  Xuexiu Zhong\thanks{School of Mathematical Sciences, South China Normal University. X. Zhong is supported by the NSFC (No.12271184) and Young Top-notch Talent Project of Guangdong Province (2024TQ08A725). E-mail: zhongxuexiu1989@163.com}, 
    \quad Wenming Zou\thanks{Department of Mathematical Sciences, Tsinghua University.  W. Zou is supported by National Key R\&D Program of China (Grant 2023YFA1010001) and NSFC (No. 12171265 and 12271184). E-mail: zou-wm@mail.tsinghua.edu.cn}
}

\begin{document}
\allowdisplaybreaks
\maketitle

\begin{abstract}
    This paper establishes a bivariate Hardy-Sobolev inequality. Let $\Omega \subset \R^N$ ($N \geq 3$) be an open domain, $s \in (0,2)$, $\alpha > 1$, $\beta > 1$ with $\alpha + \beta = 2^*(s)$, and $\kappa \in \R$. For any functions $u, v \in D_0^{1,2}(\Omega)$, we prove the inequality:
    \begin{multline*}
        \int_{\Omega} |\nabla u|^2 \, \mathrm{d}x + \int_{\Omega} |\nabla v|^2 \, \mathrm{d}x \\
        \ge S_{\alpha,\beta,\lambda,\mu}(\Omega) \left( \int_{\Omega} \Big( \lambda \frac{|u|^{2^*(s)}}{|x|^s} + \mu \frac{|v|^{2^*(s)}}{|x|^s} + 2^*(s) \kappa \frac{|u|^\alpha |v|^\beta}{|x|^s} \Big)\, \mathrm{d}x \right)^{\frac{2}{2^*(s)}}.
    \end{multline*}
    We derive the best constant $S_{\alpha,\beta,\lambda,\mu}(\Omega)$ and characterize the set of minimizers. Moreover, for $\Omega = \R^N$ and $\kappa > 0$, we obtain sharp stability results for nonnegative functions.

    \vskip0.12in

    \vskip0.12in


\end{abstract}

\newpage
\section{Introduction}
\subsection{Hardy-Sobolev Inequality and Its Stability}
The celebrated Hardy-Sobolev inequality states that for any function $u \in D_0^{1,2}(\Omega)$,
\[
    \int_{\Omega} |\nabla u|^2\,\mathrm{d}x \ge \mu_s(\Omega) \left(\int_{\Omega} \frac{|u|^{2^*(s)}}{|x|^s}\,\mathrm{d}x\right)^{\frac{2}{2^*(s)}},
\]
where $\Omega$ is an open domain in $\R^N$ ($N \geq 3$), $s \in (0,2)$, and $2^*(s) = \frac{2(N-s)}{N-2}$.  {Here, $\mu_s(\Omega)$ denotes the best constant, defined by}
\[
    \mu_s(\Omega) := \inf\left\{ \frac{\int_{\Omega} |\nabla u|^2\,\mathrm{d}x}{\left(\int_{\Omega} \frac{|u|^{2^*(s)}}{|x|^s}\,\mathrm{d}x\right)^{\frac{2}{2^*(s)}}} : u \in D_0^{1,2}(\Omega) \setminus \{0\} \right\}.
\]

When the origin $0$ lies in the interior of $\Omega$, it is well-known that $\mu_s(\Omega) = \mu_s(\R^N)$, and $\mu_s(\Omega)$ is never attained unless the capacity of $\R^N \setminus \Omega$ is zero.  However, the situation differs significantly when $\partial\Omega$ has a cusp at $0$ or when $\partial\Omega$ is smooth at $0$, compared to the non-singular case. For detailed discussions, we refer to Egnell \cite{Egnell1992} and Ghoussoub and Kang \cite{Ghoussoub2004}.

In the case $\Omega = \R^N$, Lieb \cite{Lieb1983} established that the best constant is  given explicitly by
\[
    \mu_s(\R^N) = (N-2)(N-s) \left( \frac{1}{2-s} \omega_{N-1} \frac{\Gamma^2((N-s)/(2-s))}{\Gamma((2N-2s)/(2-s))} \right)^{\frac{2-s}{N-s}},
\]
where $\omega_{N-1}$ represents the volume of the unit sphere in $\R^N$. The set of minimizers for this inequality is characterized as
\begin{equation}\label{Uk}
    \mathcal{M} = \left\{ U(k,\tau) = k k_0 \left(1 + |\tau x|^{2-s}\right)^{-\frac{N-2}{2-s}} \tau^{\frac{N-2}{2}} : k \in \R \setminus \{0\}, \, \tau > 0 \right\},
\end{equation}
where $k_0$ is a normalization constant ensuring $\int_{\R^N} \frac{1}{|x|^s}|U(1,1)|^{2^*(s)}\,\mathrm{d}x = 1$.  The stability of this inequality was subsequently proved by R\u{a}dulescu, Smets, and Willem in \cite{Radulescu2002}:
\begin{alphatheorem}\label{HS-single}
    There exists a constant $L = L(N,s) > 0$ such that for every $u \in D_{0}^{1,2}(\R^N)$,
    \begin{equation}
        \int_{\R^N} |\nabla u|^2\,\mathrm{d}x - \mu_s(\R^N) \left( \int_{\R^N} \frac{|u|^{2^*(s)}}{|x|^s}\,\mathrm{d}x \right)^{\frac{2}{2^*(s)}} \ge L \inf_{v \in \mathcal{M}} \int_{\R^N} |\nabla u - \nabla v|^2\,\mathrm{d}x.
    \end{equation}
\end{alphatheorem}

The investigation of stability for functional inequalities traces back to the work of Brezis and Lieb \cite{Brezis1985}, with the first explicit result provided by Bianchi and Egnell \cite{Bianchi1991}. In recent years, there has been growing interest in the quantitative stability of various functional and geometric inequalities. Prominent examples include the isoperimetric inequality \cite{Aubin1976a, Figalli2010, Cianchi2011}, the Brunn-Minkowski inequality \cite{Figalli2017, Figalli2024}, Sobolev-type inequalities \cite{Figalli2022, Chen2013, Gazzola2010, Neumayer2020, Talenti1976, Figalli2013, Cianchi2006, Fusco2007, Zhang2025}, the Gagliardo-Nirenberg-Sobolev inequality \cite{Carlen2013, Bonforte2025}, and the Caffarelli-Kohn-Nirenberg inequality \cite{Caffarelli1984, Cazacu2024, Lam2017, Wei2024, Zhou2024}.

\vskip0.3in

\subsection{Bivariate Hardy-Sobolev Inequality and Its Stability}
Although the stability of single-variable inequalities has been extensively studied, the corresponding theory for bivariate cases remains relatively underdeveloped. A foundational result in this direction is the bivariate inequality \eqref{doubleHLS}, which was established along with an analysis of its optimal constant, extremal functions, and regularity properties in the second and the third authors' paper \cite{Zhong+Zou=1} and the second author's doctoral dissertation \cite[Section 7.6]{Zhong2015}.

We begin by establishing the following bivariate Hardy-Sobolev inequality:
\begin{theorem}\label{HS-double}
    Let $\Omega\subset \R^N$ be an open domain, $s\in(0,2),\, \kappa\in \R,\, \alpha>1,\,\beta>1,\,\alpha+\beta=2^*(s)$, then there exists $C=C(\alpha,\beta,\lambda,\mu,\kappa,N,s)>0$ such that for every $(u,v)\in \mathscr{D}(\Omega) $, there holds
    \begin{align}\label{doubleHLS}
       & \int_{\Omega} |\nabla u|^2\,\mathrm{d}x + \int_{\Omega} |\nabla v|^2\,\mathrm{d}x  \nonumber \\
       & \ge
        C \left( \int_{\Omega} \big( \lambda \frac{|u|^{2^*(s)}}{|x|^s} + \mu \frac{|v|^{2^*(s)}}{|x|^s} + 2^*(s)\kappa \frac{|u|^\alpha|v|^\beta}{|x|^s}\big)\,\mathrm{d}x\right)^{\frac{2}{2^*(s)}}.
    \end{align}
    Here
    \[
        \mathscr{D}(\Omega) := D_0^{1,2}(\Omega)\times D_0^{1,2}(\Omega).
    \]
\end{theorem}

We first  determine the best constant of \eqref{doubleHLS} and characterize its minimizers. Define
\begin{align}\label{best constant}
    & S_{\alpha,\beta,\lambda,\mu}(\Omega) \nonumber\\
    & :=
    \inf_{(u,v)\in \mathscr{D}(\Omega)\backslash\{{(0,0)}\}}
    \frac{\int_{\Omega}(|\nabla u|^2+|\nabla v|^2)\,\mathrm{d}x}{\left( \int_{\Omega} \big( \lambda \frac{|u|^{2^*(s)}}{|x|^s} + \mu \frac{|v|^{2^*(s)}}{|x|^s} + 2^*(s)\kappa \frac{|u|^\alpha|v|^\beta}{|x|^s}\big)\,\mathrm{d}x \right)^{\frac{2}{2^*(s)}}}.
\end{align}
\begin{theorem}\label{mth1.2}
    Let $\Omega$ be an open domain in $\R^N$, $0<s<2,\alpha>1,\beta>1,\alpha+\beta=2^*(s)$.
    \begin{itemize}
        \item If $\kappa \le 0$, then $S_{\alpha,\beta,\lambda,\mu}(\Omega) = (\max\{\lambda,\mu\})^{-2/2^*(s)}\mu_s(\Omega)$, and it can only be attained by the semi-trivial pairs:
        \[ \begin{cases}
            (\phi,0), & \text{if $\lambda>\mu$;}\\
            (0,\phi), & \text{if $\lambda<\mu$;}\\
            (0,\phi) \, \text{or}\, (\phi,0), & \text{if $\lambda=\mu$};
        \end{cases}\]
        where $\phi$ is an extremal of $\mu_s(\Omega)$.
        \item If $\kappa >0$, then $S_{\alpha,\beta,\lambda,\mu}(\Omega) = \inf_{t\in (0,\infty)}g(t)\mu_s(\Omega)$, where
        \[
            g(t) = \frac{1+t^2}{\big(\lambda + \mu t^{2^*(s)} + 2^*(s)\kappa t^\beta\big)^{\frac{2}{2^*(s)}}}.
        \]
        If $(\phi,\psi)\subset \mathscr{D}(\Omega)\backslash  {\{(0,0)\}}$ is an extremal of $S_{\lambda,\mu,\alpha,\beta}(\Omega)$, then $|\psi|=|t_0\phi|$, where $t_0\in[0,\infty]$ satisfies $g(t_0)=\inf_{t\in(0,+\infty)} g(t)$ and $\phi$ is an extremal of $\mu_s(\Omega)$ (If $t_0=\infty$, then $(\phi,\psi)=(0,\psi)$ and $\psi$ is an extremal of $\mu_s(\Omega)$).

        Define the set of minimizers of $g$ by
        \begin{equation}\label{eq:20251111-1518}
            \mathcal{S} := \left\{t_0 \in [0,\infty] : g(t_0) = \inf\limits_{t \in (0,\infty)} g(t)\right\}
        \end{equation}

        Then for $\Omega=\R^N$ with $\kappa>0$, the set of minimizers of \eqref{doubleHLS} can be characterized as
        \begin{equation}
            \begin{aligned}
                \mathcal{M}' ={}& \{(U(k,\tau),\pm t_0U(k,\tau)): k\in\R\backslash\{0\},\,\tau>0,\,  {t_0 \in \mathcal{S}}\} \\
            \end{aligned}
        \end{equation}
        where $U(k,\tau)$ is defined as in \eqref{Uk}, and when $t_0=\infty$, the expression means $(0,U(k,\tau))$.
    \end{itemize}
 \end{theorem}

 Next, we focus on the case $\Omega=\R^N$, $\kappa>0$ to establish sharp stability results for this bivariate Hardy-Sobolev inequality. Define the deficit functions $\delta(u)$ and $\delta(u,v)$ as follows:
\begin{equation}
    \delta(u) := \int_{\R^N} |\nabla u|^2\,\mathrm{d}x - \mu_s(\R^N)\left( \int_{\R^N} \frac{|u|^{2^*(s)}}{|x|^s}\,\mathrm{d}x\right)^{\frac 2{2^*(s)}}
\end{equation}
for $u\in D_{0}^{1,2}(\R^N)$, and
\begin{equation}
    \begin{aligned}
        &\delta(u,v) := \int_{\R^N} |\nabla u|^2\,\mathrm{d}x + \int_{\R^N} |\nabla v|^2\,\mathrm{d}x \\
        & \quad - S_{\alpha,\beta,\lambda,\mu}(\R^N) \left( \int_{\R^N} \big( \lambda \frac{|u|^{2^*(s)}}{|x|^s} + \mu \frac{|v|^{2^*(s)}}{|x|^s} + 2^*(s)\kappa \frac{|u|^\alpha|v|^\beta}{|x|^s}\big)\,\mathrm{d}x\right)^{\frac 2{2^*(s)}}
    \end{aligned}
\end{equation}
for $(u,v)\in\mathscr{D}(\R^N)$. Let $\|\cdot\|_r$ denote the $L^r(\R^N)$-norm for $1\leq r<\infty$, and define the set of degenerate minimizers of $g$ by
\[
    \tilde{\mathcal{S}} := \mathcal{S} \cap \big( \{t\in [0,\infty): g''(t) = 0\} \cup \{\infty: \tilde{g}''(0) = 0\}\big),
\]
where $\tilde{g}$ characterizes the property of $g$ at $\infty$:
\begin{equation}\label{symmetry}
    \tilde{g}(t) := g(1/t) = \frac{1+t^2}{(\mu + \lambda t^p + \kappa p t^\alpha)^{2/p}}.
\end{equation}

Then we have the following stability theorem:

\begin{theorem}\label{main thm}
    Assume $s\in(0,2)$, $\kappa>0$, $\alpha>1$, $\beta>1$, and $\alpha+\beta = 2^*(s)$. If $(\alpha,\beta,\lambda,\mu)\neq (2,2,2\kappa,2\kappa)$, then there exists a constant $C=C(\lambda,\mu,\kappa,\alpha,\beta,N,s)>0$ such that for all nonnegative functions $(u,v)\in \mathscr{D}(\R^N)\setminus \{(0,0)\}$,
    \[
        \inf_{(U_0,V_0)\in \mathcal{M}'} \frac{\|\nabla(u-U_0)\|_2^2 + \|\nabla(v-V_0)\|_2^2}{\|\nabla u\|_2^2 + \|\nabla v\|_2^2}
        \leq C \left(\frac{\delta(u,v)}{\|\nabla u\|_2^2 + \|\nabla v\|_2^2}\right)^{\iota},
    \]
    where the exponent $\iota$ is determined as follows:

    \begin{description}
    \item[\textbf{(I)}] $\iota = 1$ if $\tilde{\mathcal{S}} = \emptyset$, i.e. all the minimizers of $g$ are nondegenerate;
    \item[\textbf{(II)}] $\iota = 1/2$ if $\tilde{\mathcal{S}} \neq \emptyset$, i.e. $g$ has a degenerate minimizer.
    \end{description}

    More precisely, \emph{\textbf{(II)}} includes the following four scenarios: [Table 1 here]
    \begin{table}[h]
        \centering
        \caption{Summary of cases in Theorem \ref{main thm} leading to $\iota = 1/2$}
        \label{tab:theorem1-3-cases}
        \begin{tabular}{|c|c|c|c|}
            \hline

            \textbf{Case}
            & \textbf{Parameter conditions}
            & $\mathcal{S}$
            & $\tilde{\mathcal{S}}$ \\

            \hline

            \textbf{(II.1)}
            & \thead{ $\alpha, \beta \in (1,2)$, \\
                $\lambda = \frac{2\kappa\alpha}{p-2} \left( \frac{2-\beta}{2-\alpha} \right)^{\frac{\beta}{2}-1}$, \\
                $\mu = \frac{2\kappa\beta}{p-2} \left( \frac{2-\beta}{2-\alpha} \right)^{1-\frac{\alpha}{2}}$ }
            & $ \left\{ \sqrt{\frac{2-\beta}{2-\alpha}} \right\} $
            & $\left\{ \sqrt{\frac{2-\beta}{2-\alpha}} \right\} $ \\

            \hline
            
            \textbf{(II.2)}
            & $\beta = 2 \leq \alpha, \, \mu < \lambda = 2\kappa$
            & $\{0\}$
            & $\{0\}$ \\

            \hline
            
            \textbf{(II.3)}
            & $\alpha = 2 \leq \beta,\, \lambda < \mu = 2\kappa$
            & $\{\infty\}$
            & $\{\infty\}$ \\
            
            \hline
            
            \textbf{(II.4)}
            & \thead{ $\min\{\alpha,\beta\} = 2< \max\{\alpha,\beta\}$, \\ $\lambda = \mu = 2\kappa$ }
            & $\{0, \infty\}$
            & \thead{ $\{0\}$, if $\beta=2$, \\ $\{\infty\}$, if $\alpha=2$ } \\

            \hline
        \end{tabular}
    \end{table}

    Moreover, the exponent $\iota$ is sharp in all cases (see Remark \ref{rmk11}).
\end{theorem}

Finally, through a transformation, we derive sharp stability for a class of bivariate Hardy-Sobolev inequalities:
\begin{corollary}\label{main_cor}
    Assume $ \ell\in(0,1),\,s\in (0,2)$, $\kappa >0,\, \alpha>1,\,\beta>1,\, \alpha+\beta=2^*(s)$. Then for any $u,v\in \mathscr{U}$, where
    \[ \mathscr{U}:= \{u \in L^{2^*(s)}(\R^N, |x|^{-(N-s)(1- \ell)-s}\,\mathrm{d}x):\nabla u \in L^2(\R^N,|x|^{-(N-2)(1- \ell)}\,\mathrm{d}x)\},\]
    there holds
    \begin{equation}\label{class1}
        \begin{aligned}
            & \int_{\R^N}|x|^{-(N-2)(1- {\ell})}(|\nabla u|^2+|\nabla v|^2)\,\mathrm{d}x    \\
            &\ge{} S_{\alpha,\beta,\lambda,\mu}(\R^N) \cdot  {\ell}^{\frac 2{2^*(s)}+1} \\
            & \quad \cdot\left(\int_{\R^N} |x|^{-(N-s)(1- {\ell})-s}(\lambda |u|^{2^*(s)}+\mu|v|^{2^*(s)}+2^*(s)\kappa |u|^\alpha|v|^\beta)\,\mathrm{d}x \right)^{\frac{2}{2^*(s)}},
        \end{aligned}
    \end{equation}
    and equality holds if and only if $(u,v)=(w,t_0w)$, where $g(t_0) = \inf_{t\in(0,\infty)}g(t)$ and $w\in \mathcal{M}_{ {\ell}}$ ($(u,v)=(0,w)$ if $t_0=\infty$):
    \[
        \mathcal{M}_{ {\ell}} := \{c(1+|\lambda x|^{(2-s) {\ell}})^{-\frac{N-2}{2-s}}\lambda^{\frac{(N-2) {\ell}}{2}}:\lambda>0,\, c\in\R\}
    .\]
    Denote by $\mathcal{M}'_{ {\ell}}$ the set of minimizers of \eqref{class1}, $\delta_{ {\ell}}(u,v)$ the deficit of inequality \eqref{class1}, then there exists $C=C(\alpha,\beta,\kappa,\lambda,\mu,s,N, {\ell})>0$ such that for any nonnegative $u,v\in \mathscr{U}$, there holds
    \begin{multline*}
        \inf_{(U_0,V_0)\in \mathcal{M}'_{ {\ell}}} \frac{\int_{\R^N} |x|^{-(N-2)(1- {\ell})}(|\nabla(u-U_0)|^2 + |\nabla(v-V_0)|^2)\,\mathrm{d}x}{\int_{\R^N} |x|^{-(N-2)( {\ell}-1)}(|\nabla u|^2+|\nabla v|^2)\,\mathrm{d}x}\\
        \le C \left(\frac{\delta_{ {\ell}}(u,v)}{\int_{\R^N} |x|^{-(N-2)( {\ell}-1)}(|\nabla u|^2+|\nabla v|^2)\,\mathrm{d}x}\right)^\iota,
    \end{multline*}
    where $\iota=1$ if \emph{\textbf{(I)}} occurs, and $\iota=1/2$ if \emph{\textbf{(II)}} occurs.
\end{corollary}

\subsection{Structure of the Paper}
In Section \ref{const}, we give the proof of Theorem \ref{mth1.2}. Then, in Section \ref{sec2}, we state and prove some basic tools that are needed for the proof of stabilities, including the spectrum analysis of the Hardy-Sobolev inequality, properties of $g(t)$ and some elementary inequalities. In Section \ref{sec3}, we state the qualitative stability result. The proofs of Theorem \ref{main thm} and Corollary \ref{main_cor} are presented in Section \ref{sec4}.

\section{The Best Constant and Minimizers}\label{const}

From now on, for simplicity, we write $p=2^*(s):=\frac{2(N-s)}{N-2}$. By  {$N\geq 3$ and $s<2$}, we always have $p>2$.

For $(u,v)\in\mathscr{D}(\Omega)\backslash\{ {(0,0)}\}$, define
\[
    G(u,v) := \frac{\int_{\Omega}(|\nabla u|^2+|\nabla v|^2)\,\mathrm{d}x}{\left( \int_{\Omega} \frac{1}{|x|^s}\big( \lambda |u|^{p} + \mu |v|^{p} + \kappa p|u|^\alpha|v|^\beta\big)\,\mathrm{d}x \right)^{2/p}},
\]
then $S_{\alpha,\beta,\lambda,\mu}(\Omega)=\inf_{(u,v)\in\mathscr{D}(\Omega)\backslash\{ {(0,0)}\}} G(u,v)$.

\begin{proof}[\bf Proof of Theorem \ref{mth1.2}]
    If $\kappa\le 0$, for any $(u,v)\in \mathscr{D}(\Omega)\backslash \{ {(0,0)}\}$,
    \begin{align*}
        G(u,v) \ge{}& \frac{\int_{\Omega} (|\nabla u|^2 + |\nabla v|^2) \,\mathrm{d}x}{\left(\int_{\Omega} \frac 1{|x|^s} \big(\lambda|u|^p + \mu |v|^p\big)\,\mathrm{d}x\right)^{2/p}} \\
        \ge{}& (\max\{\lambda,\mu\})^{-2/p} \frac{\int_{\Omega} (|\nabla u|^2 + |\nabla v|^2) \,\mathrm{d}x}{\left(\int_{\Omega} \frac 1{|x|^s} \big(|u|^p + |v|^p\big)\,\mathrm{d}x\right)^{2/p}} \\
        \ge{}& (\max\{\lambda,\mu\})^{-2/p}\mu_s(\Omega) \frac{ \left(\int_\Omega \frac{|u|^p}{|x|^s}\,\mathrm{d}x\right)^{2/p} + \left(\int_\Omega \frac{|v|^p}{|x|^s}\,\mathrm{d}x\right)^{2/p}}{\left(\int_{\Omega} \big(\frac{|u|^p}{|x|^s} + \frac{|v|^p}{|x|^s}\big)\,\mathrm{d}x\right)^{2/p}} \\
        \ge{}& (\max\{\lambda,\mu\})^{-2/p}\mu_s(\Omega).
    \end{align*}
    In the last step, we use the following elementary inequality: for $0<\iota<1$,
    \[ x^\iota + y^\iota \ge (x+y)^{\iota}, \quad \forall\, x,y\ge 0,\]
    and equality holds if and only if one of $x,y$ is zero. By taking the infimum of $G(u,v)$ over $\mathscr{D}(\Omega)\backslash\{ {(0,0)}\}$, we obtain that $S_{\alpha,\beta,\lambda,\mu}(\Omega) \ge (\max\{\lambda,\mu\})^{-2/p}\mu_s(\Omega)$. Since
    \[ \inf_{u\in D_0^{1,2}(\Omega)\backslash\{0\}} G(u,0) = \lambda^{-2/p} \inf_{u\in D_0^{1,2}(\Omega)\backslash\{0\}} \frac{\int_\Omega |\nabla u|^2 \,\mathrm{d}x}{\left(\int_\Omega \frac{1}{|x|^s}|u|^p\,\mathrm{d}x\right)^{2/p}} = \lambda^{-2/p}\mu_s(\Omega),\]
    \[ \inf_{v\in D_0^{1,2}(\Omega)\backslash\{0\}} G(0,v) = \mu^{-2/p}\inf_{v\in D_0^{1,2}(\Omega)\backslash\{0\}} \frac{\int_\Omega |\nabla v|^2 \,\mathrm{d}x}{\left(\int_\Omega \frac{1}{|x|^s}|v|^p\,\mathrm{d}x\right)^{2/p}} = \mu^{-2/p}\mu_s(\Omega),\]
    we derive that $S_{\alpha,\beta,\lambda,\mu}(\Omega)=(\max\{\lambda,\mu\})^{-2/p}\mu_s(\Omega)$. Moreover, by the process of the proof, $S_{\alpha,\beta,\lambda,\mu}(\Omega)$ is only achieved by $(\phi,0)$ if $\lambda>\mu$ (resp. $(0,\phi)$ if $\lambda<\mu$), where $\phi$ is an extremal of $\mu_s(\Omega)$.

    If $\kappa > 0$. For any $u\neq 0, v\neq 0$ and $t\geq 0$,  we have
    \begin{equation}\label{2014-11-28-xbue3}
        G(u,tv)=\frac{\int_\Omega \big(|\nabla u|^2+|\nabla v|^2t^2\big)\,\mathrm{d}x}{\Big(\int_\Omega \frac{1}{|x|^s}\big(\lambda |u|^{p} + \mu |v|^{p}t^p + \kappa p|u|^\alpha |v|^\beta t^\beta\big)\, \mathrm{d}x\Big)^{2/p}}
    \end{equation}
    Hence,
    \begin{align*}
        G(u,tu) ={}& \frac{1+t^2}{\left(\lambda+\mu t^{p}+\kappa pt^\beta\right)^{2/p}} \frac{\int_\Omega |\nabla u|^2 \,\mathrm{d}x}{\left(\int_\Omega \frac{|u|^{p}}{|x|^s}\, \mathrm{d}x\right)^{2/p}}\\
        ={}&g(t)\frac{\int_\Omega |\nabla u|^2 \,\mathrm{d}x}{\left(\int_\Omega \frac{1}{|x|^s}|u|^p \,\mathrm{d}x\right)^{2/p}}.
    \end{align*}
    Notice that
    \[ \lim_{t\to \infty} G(u,tu) = g(\infty) \frac{\int_\Omega |\nabla u|^2 \,\mathrm{d}x}{\left(\int_\Omega \frac{|u|^{p}}{|x|^s}\, \mathrm{d}x\right)^{2/p}} = G(0,u).\]
    Thus,
    \begin{align}\label{2014-11-28-xbue8}
        S_{\alpha,\beta,\lambda,\mu}(\Omega)={}&\inf_{(u,v)\in \mathscr{D}(\Omega)\backslash\{{(0,0)}\}}G(u,v)\nonumber\\
        \leq&\inf_{u\in D_{0}^{1,2}(\Omega){\backslash \{0\}}}\inf_{t\in (0,+\infty)}G(u,tu)\nonumber\\
        ={}&\inf_{t\in (0,+\infty)}g(t) \inf_{u\in D_{0}^{1,2}(\Omega)\backslash \{0\}}\frac{\int_\Omega |\nabla u|^2 \,\mathrm{d}x}{\left(\int_\Omega \frac{1}{|x|^s}|u|^p\,\mathrm{d}x\right)^{2/p}}\nonumber\\
        ={}&\inf_{t\in (0,+\infty)}g(t) \mu_s(\Omega).
    \end{align}

    To show $S_{\alpha,\beta,\lambda,\mu}(\Omega) = \inf_{t\in (0,+\infty)}g(t) \mu_s(\Omega)$, we only need to prove the reverse inequality. Let $\{(u_n,v_n)\}$ be a minimizing sequence of $S_{\alpha,\beta,\lambda,\mu}(\Omega)$. Since $G(u,v)=G(tu,tv)$ for all $t>0$, without loss of generality, we may assume that
    $$\int_\Omega \frac{1}{|x|^s}\big(|u_n|^{p}+ |v_n|^{p}\big)\,\mathrm{d}x=1,\quad \forall\, n\ge 1,$$
    and $$G(u_n,v_n)=S_{\alpha,\beta,\lambda,\mu}(\Omega)+o(1).$$

    \noindent\textbf{
        Case 1. $\liminf\limits_{n\rightarrow +\infty}\int_\Omega \frac{|u_n|^{p}}{|x|^s}\, \mathrm{d}x=0$.
    }
    Since $\{v_n\}$ is bounded in $L^{p}(\Omega,\frac{\mathrm{d}x}{|x|^s})$, by H\"older inequality, up to a subsequence,
    \[ 
        \int_\Omega \big(\lambda \frac{|u_n|^{p}}{|x|^s}+\mu \frac{|v_n|^{p}}{|x|^s}+\kappa p\frac{|u_n|^\alpha |v_n|^\beta}{|x|^s}\big)\,\mathrm{d}x=\mu\int_\Omega \frac{|v_n|^{p}}{|x|^s}\,\mathrm{d}x+o(1)=\mu+o(1).
    \]
    Hence,
    \[
        \lim_{n\rightarrow \infty}G(u_n,v_n)\geq \lim_{n\rightarrow \infty} G(0,v_n)
    \]
    and $(0,v_n)$ is also a minimizing sequence of $S_{\alpha,\beta,\lambda,\mu}(\Omega)$. As a result,
    \begin{align*}
        S_{\alpha,\beta,\lambda,\mu}(\Omega)=\mu^{-2/p}\mu_s(\Omega) = g(\infty)\mu_s(\Omega) \geq \inf_{t\in (0,+\infty)}g(t) \mu_s(\Omega).
    \end{align*}

    \noindent\textbf{
        Case 2. $\liminf\limits_{n\rightarrow +\infty}\int_\Omega \frac{|v_n|^{p}}{|x|^s}\, \mathrm{d}x=0$.
    }
    Similarly, we have
    \begin{align*}
        S_{\alpha,\beta,\lambda,\mu}(\Omega)=\lambda^{-2/p}\mu_s(\Omega) =  g(0)\mu_s(\Omega) \geq \inf_{t\in (0,+\infty)}g(t) \mu_s(\Omega).
    \end{align*}

    \noindent\textbf{
        Case 3. Up to a subsequence, $\exists\, \delta>0$ s.t.
        \[ 
            \lim_{n\rightarrow +\infty}\int_\Omega \frac{|u_n|^{p}}{|x|^s} \,\mathrm{d}x=\delta>0,\quad \displaystyle\lim_{n\rightarrow +\infty}\int_\Omega \frac{|v_n|^{p}}{|x|^s}\, \mathrm{d}x=1-\delta>0.
        \]
    }
    Let $t_n>0$ be that $$\int_\Omega \frac{|v_n|^{p}}{|x|^s}\,\mathrm{d}x=\int_\Omega \frac{|t_nu_n|^{p}}{|x|^s}\,\mathrm{d}x,$$ then $\{t_n\}$ is bounded and away from $0$. Up to a subsequence, we may assume that $t_n\rightarrow t_0=\left(\frac{1-\delta}{\delta}\right)^{\frac{1}{p}}$. Let $w_n=\frac{v_n}{t_n}$, then
    \begin{equation*}
        \int_\Omega \frac{|u_n|^{p}}{|x|^s}\,\mathrm{d}x=\int_\Omega\frac{|w_n|^{p}}{|x|^s}\,\mathrm{d}x
    \end{equation*}
    and by Young's inequality,
    \begin{align*}
        \int_\Omega \frac{|u_n|^\alpha|w_n|^\beta}{|x|^s}\,\mathrm{d}x\leq{}& \frac{\alpha}{p}\int_\Omega \frac{|u_n|^{p}}{|x|^s}\,\mathrm{d}x+\frac{\beta}{p}\int_\Omega \frac{|w_n|^{p}}{|x|^s}\,\mathrm{d}x\\
        ={}&\int_\Omega \frac{|u_n|^{p}}{|x|^s}\,\mathrm{d}x=\int_\Omega\frac{|w_n|^{p}}{|x|^s}\,\mathrm{d}x.
    \end{align*}
    Hence,
    \begin{align}\label{S lower bound}
        G(u_n,v_n)={}&G(u_n,t_nw_n)\nonumber\\
        ={}&\frac{\int_\Omega |\nabla u_n|^2\,\mathrm{d}x}{\Big(\int_\Omega \big(\lambda \frac{|u_n|^{p}}{|x|^s}+\mu t_{n}^{p}\frac{|w_n|^{p}}{|x|^s}+\kappa pt_n^\beta\frac{|u_n|^\alpha |w_n|^\beta}{|x|^s}\big)\,\mathrm{d}x\Big)^{2/p}}\nonumber\\
        &+\frac{\int_\Omega t_n^2|\nabla w_n|^2\,\mathrm{d}x}{\Big(\int_\Omega \big(\lambda \frac{|u_n|^{p}}{|x|^s}+\mu t_{n}^{p}\frac{|w_n|^{p}}{|x|^s}+\kappa p t_n^\beta\frac{|u_n|^\alpha |w_n|^\beta}{|x|^s}\big)\,\mathrm{d}x\Big)^{2/p}}\nonumber\\
        \geq{}&\frac{1}{\left(\lambda+\mu t_{n}^{p}+\kappa p t_n^\beta\right)^{2/p}}\frac{\int_\Omega |\nabla u_n|^2 \,\mathrm{d}x}{\left(\int_\Omega \frac{|u_n|^{p}}{|x|^s}\,\mathrm{d}x\right)^{2/p}}\nonumber\\
        &+\frac{t_n^2}{\left(\lambda+\mu t_{n}^{p}+\kappa p t_n^\beta\right)^{2/p}}\frac{\int_\Omega |\nabla w_n|^2 \,\mathrm{d}x}{\left(\int_\Omega \frac{|w_n|^{p}}{|x|^s}\,\mathrm{d}x\right)^{2/p}}\nonumber\\
        \geq{}&\frac{1}{\left(\lambda+\mu t_{n}^{p}+\kappa p t_n^\beta\right)^{2/p}}\mu_{s}(\Omega)+\frac{t_n^2}{\left(\lambda+\mu t_{n}^{p}+\kappa pt_n^\beta\right)^{2/p}}\mu_{s}(\Omega)\nonumber\\
        ={}&g(t_n)\mu_{s}(\Omega).
    \end{align}
    Letting $n\rightarrow +\infty$, we obtain that
    \[ 
        S_{\alpha,\beta,\lambda,\mu}(\Omega)\geq g(t_0)\mu_s(\Omega)\geq \inf_{t\in (0,+\infty)}g(t) \mu_s(\Omega).
    \]
    Thereby $\displaystyle S_{\alpha,\beta,\lambda,\mu}(\Omega)= \inf_{t\in (0,+\infty)}g(t) \mu_s(\Omega)$ is proved.

    If $(\phi,\psi)$ is an extremal of $S_{\alpha,\beta,\lambda,\mu}(\Omega)$, let $(u_n,v_n)\equiv (\phi,\psi)$, then each step in \eqref{S lower bound} takes equality and {$t_0\in \mathcal{S}$}. Hence $|\psi| = |t_0\phi|$ and $\phi$ is an extremal of $\mu_s(\Omega)$ (If $t_0=\infty$, then $\phi=0$ and $\psi$ is an extremal). In particular, if $\Omega=\R^N$, then minimizers do not change sign, implying that $\psi = \pm t_0\phi$.
\end{proof}

\section{Preparations for the Proof of Stability} \label{sec2}
From now on, we omit $\R^N$ if there are no confusions. For $A,B\in\R$, we say $A\lesssim B$ if there exists a constant $C>0$ depending only on $\alpha,\beta,\lambda,\mu,N,s$ such that $A\le CB$. We say $A\approx B$ if both $A\lesssim B$ and $B\lesssim A$.

\subsection{Spectrum Analysis of the Hardy-Sobolev Inequality}
For $V\in\mathcal{M}\backslash\{0\}$, consider the eigenvalue problem
\begin{equation}\label{eigenprob}
    \begin{cases}
        -\Delta u = \lambda |x|^{-s}V^{p-2}u & \text{in $\R^N$},\\
        u \in D^{1,2}(\R^N).
    \end{cases}
\end{equation}
\begin{lemma}\label{spectrum-HLS}
    The first two eigenvalues of \eqref{eigenprob} are given by $\Lambda_1=\mu_s$ and $\Lambda_2 = (p-1)\mu_s$. The eigenspaces are spanned by $V$ and $\partial_\lambda V$, respectively.
\end{lemma}
For the proof of it, see \cite[Proposition 4.6]{Smets2003}.

\subsection{Properties of $g(t)$}
{Let $\mathcal{S}$ be defined by \eqref{eq:20251111-1518}.
Next, we construct equality and inequality of any $t_0\in \mathcal{S}$.}
\begin{proposition} \label{prop g(t)}
    If  {$t_0\in \mathcal{S}\backslash\{0,\infty\}$}, it satisfies
    \begin{equation}\label{g'=0}
        \lambda - \mu t_0^{p-2} + \kappa\alpha t_0^{\beta} - \kappa \beta t_0^{\beta-2} = 0,
    \end{equation}
    and
    \begin{equation}\label{g''>=0}
        \alpha(2-\alpha)\kappa t_0^\beta + \kappa t_0^{\beta-2}\alpha\beta \ge (p-2)\lambda.
    \end{equation}
    We can additionally get
    \begin{equation}\label{g'=0 extend}
        \lambda + \mu t_0^p + \kappa p t_0^\beta = (1+t_0^2)(\lambda + \kappa \alpha t_0^\beta).
    \end{equation}
    Moreover, $t_0$ is degenerate, i.e.,
    the equality of \eqref{g''>=0} holds, if and only if one of the following happens:
    \begin{enumerate}[(i)]
        \item $(\alpha,\beta,\lambda,\mu)=(2,2,2\kappa,2\kappa)$. And $g(t)\equiv \mathrm{const.}$ in such a case.
        \item $\alpha,\beta\in (1,2)$ and
        \[ t_0 = \sqrt{\frac{2-\beta}{2-\alpha}},\quad \lambda = \frac{2\kappa \alpha}{p-2}t_0^{\beta-2},\quad \mu = \frac{2\kappa \beta}{p-2}t_0^{2-\alpha}.\]
        In particular, $g^{(i)}(t_0)=0$ for $i=1,2,3$ and $g^{(4)}(t_0)>0$ in such a case.
    \end{enumerate}
\end{proposition}
\begin{proof}
    We compute the derivatives of $g$:
    \begin{align*}
        g'(t) ={}& \frac{2t(\lambda + \mu t^p + \kappa p t^\beta) - 2(1+t^2)(\mu t^{p-1} + \kappa \beta t^{\beta-1})}{(\lambda + \mu t^p + \kappa p t^\beta)^{2/p + 1}} \\
        ={}& \frac{2t\lambda + 2\kappa \alpha t^{\beta+1} - 2\mu t^{p-1} -2\kappa \beta t^{\beta-1}}{(\lambda + \mu t^p + \kappa p t^\beta)^{2/p+1}} \\
        ={}& \frac{2t}{(\lambda + \mu t^p + \kappa p t^\beta)^{2/p+1}} (\lambda - \mu t^{p-2} + \kappa \alpha t^\beta - \kappa \beta t^{\beta-2}).
    \end{align*}
    Since $t_0\in (0,\infty)$ is a minimizer of $g$, we have $g'(t_0) =0$ and derive \eqref{g'=0}. Thus \eqref{g'=0 extend} holds:
    \[ 
        \lambda + \mu t_0^p + \kappa p t_0^\beta = \lambda + \kappa p t_0^{\beta} + t_0^2(\lambda + \kappa\alpha t_0^\beta - \kappa\beta t_0^{\beta-2}) = (1+t_0^2)(\lambda + \kappa \alpha t_0^\beta).
    \]
    Next, we compute $g''(t)$. Write $h(t) = \frac{2t}{(\lambda + \mu t^p + \kappa p t^\beta)^{2/p + 1}}$ and $r(t) = \lambda - \mu t^{p-2} + \kappa \alpha t^{\beta} - \kappa \beta t^{\beta-2}$, then
    \begin{equation}\label{eq:20251110-1746}
        g''(t) = h'(t)r(t) + h(t)r'(t).
    \end{equation}
    Since $h$ is smooth in $(0,\infty)$, and $r(t_0)=0$, $h(t_0)>0$, by the fact that $g''(t_0)\ge 0$ and substitute $\mu$ by \eqref{g'=0}, we obtain \eqref{g''>=0}:
    \begin{align*}
        0 \le{}& g''(t_0) = h(t_0)r'(t_0) = h(t_0)t^{-1}(-(p-2)\mu t_0^{p-2} + \kappa\alpha\beta t_0^\beta - \kappa\beta(\beta-2)t_0^{\beta-2}) \\
        \implies 0 \le{}&  -(p-2)\mu t_0^{p-2} + \kappa\alpha\beta t_0^\beta - \kappa\beta(\beta-2)t_0^{\beta-2}\\
        ={}& -(p-2)\lambda + \alpha(2-\alpha )\kappa t_0^\beta + \alpha\beta\kappa t_0^{\beta-2}.
    \end{align*}
    The only thing left is to figure out when $g''(t_0)=0$ happens. In fact, if $g''(t_0)=0$, then by Taylor's expansion,
    \[ 
        g(t) = g(t_0) + \frac 16 g^{(3)}(t_0)(t-t_0)^3 + \frac{1}{24}g^{(4)}(t+\theta(t-t_0))(t-t_0)^4.
    \]
    Since $t_0$ is a minimizer of $g(t)$, and $t_0\in (0,+\infty)$ is an interior point, we must have $g^{(3)}(t_0)=0$ and $g^{(4)}(t_0)\ge 0$. Notice that $r(t_0) = 0$ and $h(t_0)>0$, \eqref{eq:20251110-1746} gives that   $r'(t_0)=0$ when $g''(t_0)=0$. Substituting $\lambda,\mu$ by \eqref{g'=0}, $\eqref{g''>=0}$ yields
    \begin{align*}
        0={}& g^{(3)}(t_0) = h''(t_0)r(t_0) + 2 h'(t_0)r'(t_0) + h(t_0)r''(t_0) = h(t_0)r''(t_0) \\
        \implies 0 ={}& t_0^2r''(t_0)\\
        ={}& -(p-2)(p-3)\mu t^{p-2} + \kappa\alpha\beta(\beta-1)t_0^{\beta} - \kappa\beta(\beta-2)(\beta-3)t_0^{\beta-2} \\
        ={}& -(p-2)(p-3)\lambda - (p-2)(p-3)\kappa\alpha t_0^\beta + (p-2)(p-3)\kappa\beta t_0^{\beta-2} \\
        & + \beta(\beta-1)\kappa\alpha t_0^\beta- (\beta-2)(\beta-3)\kappa\beta t_0^{\beta-2} \\
        ={}& -(p-3)((2-\alpha)\kappa\alpha t_0^\beta + \alpha\kappa\beta t_0^{\beta-2}) \\
        & - (p-2)(p-3)\kappa\alpha t_0^\beta + (p-2)(p-3)\kappa\beta t_0^{\beta-2} \\
        & + \beta(\beta-1)\kappa\alpha t_0^\beta - (\beta-2)(\beta-3)\kappa\beta t_0^{\beta-2} \\
        ={}& \beta(2-\alpha)\kappa\alpha t_0^{\beta} + (\beta-2)\alpha \kappa\beta t_0^{\beta-2}.
    \end{align*}
    Hence,
    \begin{equation}\label{alpha beta 3}
        (2-\alpha)t_0^2 + (\beta-2) = 0.
    \end{equation}
    Taking \eqref{alpha beta 3} into \eqref{g''>=0}, we obtain
    \begin{equation}\label{eq:20251110-1801}
        (p-2)\lambda = \kappa\alpha t_0^{\beta-2} ( (2-\alpha)t_0^2 + \beta) = 2\kappa \alpha t_0^{\beta-2}.
    \end{equation}
    If $\alpha=2$, then \eqref{alpha beta 3} gives that $\beta=2$ and $p=4$. In such a case, \eqref{eq:20251110-1801} gives that $\lambda=2\kappa$. Furthermore, by \eqref{g'=0}, we obtain that $\lambda-\mu t_0^2+2\kappa t_0^2 -2\kappa=0$, and thus $(2\kappa-\mu)t_0^2=0$. Since $t_0>0$, we also have that $\mu=2\kappa$. Hence, $g(t) \equiv (2\kappa)^{-2/p}$ is constant.

    Otherwise $\alpha\neq 2$, we see from \eqref{alpha beta 3} that $\beta\neq 2$ either.
    Then \eqref{alpha beta 3} gives that
    \begin{equation}\label{eq:20251110-1850}
        t_0 = \sqrt{\frac{2-\beta}{2-\alpha}}.
    \end{equation}
    By \eqref{eq:20251110-1801}, we have that
    \begin{equation}\label{eq:20251110-1851}
        \lambda = \frac{2\kappa \alpha}{p-2}t_0^{\beta-2}=\frac{2\kappa \alpha}{p-2} \left(\frac{2-\beta}{2-\alpha}\right)^{\frac{\beta-2}{2}}.
    \end{equation}
    We remark that a direct computation shows that $2\alpha-(p-2)\beta +(p-2)\alpha t_0^2=2\beta t_0^2$, or saying that
    \begin{equation}\label{eq:20251110-1923}
        2\alpha t_{0}^{-2}-(p-2)\beta t_{0}^{-2}+(p-2)\alpha=2\beta.
    \end{equation}
    So, by \eqref{g'=0} and \eqref{eq:20251110-1923},
    \begin{align*}
        \mu t_{0}^{p-2}={}&\lambda+\kappa \alpha t_0^\beta -\kappa \beta t_{0}^{\beta-2}\\
        ={}&\frac{2\kappa \alpha}{p-2} t_{0}^{\beta-2}+\kappa \alpha t_0^\beta -\kappa \beta t_{0}^{\beta-2}\\
        ={}&\frac{\kappa  t_0^\beta}{p-2}\left[2\alpha t_{0}^{-2}-(p-2)\beta t_{0}^{-2}+(p-2)\alpha\right]\\
        ={}&\frac{2\kappa  \beta t_0^\beta}{p-2},
    \end{align*}
    which gives that
    \begin{equation}\label{eq:20251110-1919}
        \mu = \frac{2\kappa \beta}{p-2}t_0^{2-\alpha}=\frac{2\kappa \beta}{p-2} \left(\frac{2-\alpha}{2-\beta}\right)^{\frac{\alpha-2}{2}}.
    \end{equation}
    Now, we can deduce that
    \begin{equation}\label{eq:20251110-1959}
        \begin{aligned}
            & t_0^3 r^{(3)}(t_0)\\
            &= -(p-2)(p-3)(p-4)\mu t_0^{p-2} \\
            &\qquad +\kappa \alpha\beta(\beta-1)(\beta-2)t_0^\beta - \kappa\beta(\beta-2)(\beta-3)(\beta-4)t_0^{\beta-2}\\
            &=-(p-3)(p-4) 2\kappa \beta t_0^\beta  + \kappa \alpha\beta(\beta-1)(\beta-2)t_0^\beta \\
            &\qquad - \kappa\beta(\beta-2)(\beta-3)(\beta-4)t_0^{\beta-2}\\
            &=\kappa\beta t_0^\beta\left[-2(p-3)(p-4)+\alpha(\beta-1)(\beta-2)-(\beta-2)(\beta-3)(\beta-4) t_{0}^{-2}\right]\\
            &=\kappa\beta t_0^\beta\left[-2(\alpha+\beta-3)(\alpha+\beta-4)+\alpha(\beta-1)(\beta-2)+(2-\alpha)(\beta-3)(\beta-4)\right]\\
            &= 2\kappa\alpha\beta t_0^\beta (2-\alpha).
        \end{aligned}
    \end{equation}
    Recalling that $g'(t_0)=g''(t_0)=g'''(t_0)=0$ and $h(t_0)>0$, we obtain that $r(t_0) = r'(t_0) = r''(t_0) =0$.  And thus
    \begin{align*}
        0 \leq g^{(4)}(t_0) ={}& h^{(3)}(t_0)r(t_0) + 3 h''(t_0)r'(t_0) + 3h'(t_0)r''(t_0) + h(t_0) r^{(3)}(t_0) \\
        ={}& h(t_0)r^{(3)}(t_0).
    \end{align*}
    Combining with \eqref{eq:20251110-1959}, we obtain that $2\kappa\alpha\beta t_0^\beta (2-\alpha)\geq 0$. Since $\alpha\neq 2$, we see that $\alpha<2$. By \eqref{eq:20251110-1850}, we also have that $\beta<2$. That is, $\alpha,\beta\in (1,2)$ and $g^{(4)}(t_0)=h(t_0)r^{(3)}(t_0)=2\kappa\alpha\beta h(t_0) t_{0}^{\beta-3} (2-\alpha)>0$ in such a case.
\end{proof}

\begin{proposition}\label{prop g(0)}
    Suppose $0\in \mathcal{S}$. Then $\beta \ge 2$ and $\lambda \ge \mu$. Moreover, the following hold:
    \begin{enumerate}
        \item $g''(0) > 0$ if and only if either
        \begin{enumerate}
            \item $\beta > 2$, or
            \item $\beta = 2$ and $\lambda > 2\kappa$.
        \end{enumerate}
        \item $g''(0) = 0$ if and only if $\beta = 2$ and $\lambda= 2\kappa$. In this case, $g^{(3)}(0) = 0$, and $g^{(4)}(0) > 0$ if and only if either
        \begin{enumerate}
            \item $\alpha > 2$, or
            \item $\alpha = 2$ and $\mu < \lambda$.
        \end{enumerate}
        \item $g''(0)=g^{(4)}(0) = 0$ occurs only when $(\alpha, \beta, \lambda, \mu) = (2, 2, 2\kappa, 2\kappa)$, in which case $g(t) \equiv (2\kappa)^{-2/p}$ is constant.
    \end{enumerate}
    
    By symmetry, if $t_0 = \infty$ is a minimizer of $g(t)$, analogous statements hold for $\tilde{g}''(0)$ and $\tilde{g}^{(4)}(0)$ after interchanging $\alpha$ with $\beta$ and $\lambda$ with $\mu$. (Recall \eqref{symmetry} for the definition of $\tilde{g}(t)$.)
\end{proposition}

\begin{proof}
    Since $t_0 = 0$ minimizes $g(t)$, we have $g(0) \le g(\infty)$, which implies $\lambda \ge \mu$. Define
    \[
        \bar{h}(t) = \frac{2}{(\lambda + \mu t^p + \kappa p t^\beta)^{2/p + 1}}, \quad
        \bar{r}(t) = \lambda t - \mu t^{p-1} + \kappa \alpha t^{\beta+1} - \kappa \beta t^{\beta-1},
    \]
    so that $g'(t) = \bar{h}(t) \bar{r}(t)$. Clearly, $g'(0) = 0$. To compute $g''(0)$, observe:
    \begin{align*}
        g''(0) &= \bar{h}(0) \bar{r}'(0) \\
        &= 2\lambda^{-2/p-1} \left.\frac{d}{dt}\left[\lambda - \mu(p-1)t^{p-2} + \kappa\alpha(\beta+1)t^\beta - \kappa\beta(\beta-1)t^{\beta-2}\right]\right|_{t=0}.
    \end{align*}
    Evaluating the derivative at $t = 0$ yields:
    \[
        g''(0) =
        \begin{cases}
            -\infty, & \text{if } \beta \in (1,2), \\
            2\lambda^{-2/p}, & \text{if } \beta > 2, \\
            2\lambda^{-2/p-1}(\lambda - 2\kappa), & \text{if } \beta = 2.
        \end{cases}
    \]
    Since $g(0)$ is a minimizer, $g''(0) \ge 0$. Thus, either $\beta > 2$, or $\beta = 2$ with $\lambda \ge 2\kappa$. Specifically:
    \begin{itemize}
        \item $g''(0) > 0$ iff $\beta > 2$ or ($\beta = 2$ and $\lambda > 2\kappa$).
        \item $g''(0) = 0$ iff $\beta = 2$ and $\lambda = 2\kappa$.
    \end{itemize}

    Now assume $\beta = 2$ and $\lambda = 2\kappa$. Then $\bar{r}(t) = -\mu t^{\alpha+1} + \kappa\alpha t^3$. Since $\bar{r}''(0) = 0$, we have $g^{(3)}(0) = \bar{h}(0)\bar{r}''(0) = 0$. For the fourth derivative:
    \begin{align*}
        g^{(4)}(0) &= \bar{h}(0) \bar{r}^{(3)}(0) \\
        &= 2\lambda^{-2/p-1} \left.\frac{d^3}{dt^3}\left[-\mu t^{\alpha+1} + \kappa\alpha t^3\right]\right|_{t=0} \\
        &=
        \begin{cases}
            -\infty, & \text{if } \alpha \in (1,2), \\
            6\alpha\lambda^{-2/p}, & \text{if } \alpha > 2, \\
            12\lambda^{-2/p-1}(\lambda - \mu), & \text{if } \alpha = 2.
        \end{cases}
    \end{align*}
    Again, since $g(0)$ is a minimizer, $g^{(4)}(0) \ge 0$. Thus:
    \begin{itemize}
        \item $g^{(4)}(0) > 0$ iff $\alpha > 2$ or ($\alpha = 2$ and $\mu < \lambda$).
        \item $g^{(4)}(0) = 0$ iff $\alpha = 2$ and $\mu = \lambda$. In this case, $(\alpha, \beta, \lambda, \mu) = (2, 2, 2\kappa, 2\kappa)$, and $g(t) \equiv (2\kappa)^{-2/p}$ is constant. \qedhere
    \end{itemize}
\end{proof}

\begin{proposition}\label{prop:finite-minimizer}
    The set $\mathcal{S}$ is finite, except in the special case $(\alpha, \beta, \lambda, \mu) = (2, 2, 2\kappa, 2\kappa)$. Furthermore, if $(\alpha, \beta, \lambda, \mu) \neq (2, 2, 2\kappa, 2\kappa)$ and  $\tilde{\mathcal{S}}  \neq \emptyset$ , then $\mathcal{S}$ is a singleton  or doubleton. More precisely, this occurs exactly in the following four scenarios:
    \begin{description}
        \item[(i)] $\mathcal{S} = \left\{\sqrt{\frac{2-\beta}{2-\alpha}}\right\}$ with $\alpha, \beta \in (1,2)$,
        $\lambda = \frac{2\kappa \alpha}{p-2}\left(\frac{2-\beta}{2-\alpha}\right)^{\frac{\beta}{2}-1}$, $\mu = \frac{2\kappa \beta}{p-2}\left(\frac{2-\beta}{2-\alpha}\right)^{1-\frac{\alpha}{2}}$;
        \item[(ii)] $\mathcal{S} = \{0\}$ with $\beta = 2\le \alpha$, $\mu<\lambda = 2\kappa$;
        \item[(iii)] $\mathcal{S} = \{\infty\}$ with $\alpha = 2\le \beta$, $\lambda<\mu = 2\kappa$;
        \item[(iv)] $\mathcal{S} = \{0,\infty\}$ with $\min\{\alpha,\beta\} = 2 < \max\{\alpha,\beta\}, \,\lambda=\mu=2\kappa$.
    \end{description}
\end{proposition}

\begin{proof}
    The set of minimizers satisfies the inclusion:
    \[
        \mathcal{S} \subset \left\{t \in (0,\infty) : g'(t) = 0\right\} \cup \{0, \infty\}.
    \]
    Since $g'(t) = 0$ is equivalent to
    \[
        \lambda - \mu t^{p-2} + \kappa \alpha t^\beta - \kappa \beta t^{\beta-2} = 0,
    \]
    it follows that
    \[
        \mathcal{S} \subset \left\{t \in (0,\infty) : \lambda - \mu t^{p-2} + \kappa \alpha t^\beta - \kappa \beta t^{\beta-2} = 0\right\} \cup \{0, \infty\}.
    \]

    Define the function
    \[
        f(t) := \lambda - \mu t^{p-2} + \kappa \alpha t^\beta - \kappa \beta t^{\beta-2}.
    \]
    We analyze two cases based on whether $f(t)$ is constant.

    \noindent{
        \textbf{Case 1: $f(t)$ is constant.}
    }
    For $f(t)$ to be constant, the coefficients of all powers of $t$ must vanish. This requires $\beta-2=0$, $p-2=\beta$, $\kappa\alpha=\mu$.
    These conditions imply $(\alpha, \beta, \mu) = (2, 2, 2\kappa)$. In this case,
    \[
        f(t) \equiv \lambda - \kappa \beta = \lambda - 2\kappa.
    \]
    If $\lambda \neq 2\kappa$, then $f(t)$ is a nonzero constant and has no zeros. Consequently, $\mathcal{S} \subset \{0, \infty\}$, which is finite. The exceptional case occurs when $\lambda = 2\kappa$, i.e., $(\alpha, \beta, \lambda, \mu) = (2, 2, 2\kappa, 2\kappa)$.

    \noindent{
        \textbf{Case 2: $f(t)$ is not constant.}
    }
    Then $f(t)$ is an analytic function on $(0,\infty)$, so its zeros are isolated. To prove that $\mathcal{S}$ is finite, it suffices to show that the zeros of $f(t)$ are bounded away from both $0$ and $\infty$.

    \textbf{Behavior as $t \to \infty$}:
    \begin{itemize}
        \item If $\alpha \neq 2$, the dominant term as $t \to \infty$ is either $-\mu t^{p-2}$ (if $\alpha > 2$) or $\kappa \alpha t^\beta$ (if $\alpha < 2$). In both cases, $|f(t)| \to \infty$.
        \item If $\alpha = 2$ and $\kappa \alpha \neq \mu$, the dominant term is $(\kappa \alpha - \mu)t^\beta$, and $|f(t)| \to \infty$.
        \item If $\alpha = 2$ and $\kappa \alpha = \mu$, then
        \[
            f(t) = \lambda - \kappa \beta t^{\beta-2}.
        \]
        Since $f(t)$ is not constant, $\beta \neq 2$. Then $f(t)$ has a unique positive zero $t = \left(\frac{\lambda}{\kappa \beta}\right)^{\frac{1}{\beta-2}}$.
    \end{itemize}

    \textbf{Behavior as $t \to 0^+$:}
    \begin{itemize}
        \item If $\beta > 2$, or if $\beta = 2$ and $\lambda \neq 2\kappa$, then $f(0^+) \neq 0$. By continuity, there exists $t^* > 0$ such that $f(t)$ has the same sign as $f(0^+)$ for $t \in (0, t^*)$.
        \item If $\beta = 2$ and $\lambda = 2\kappa$, then $f(0^+) = 0$. In this case,
        \[
            f(t) = -\mu t^\alpha + \kappa \alpha t^2.
        \]
        This function has a unique positive zero $t = \left(\frac{\kappa \alpha}{\mu}\right)^{\frac{1}{\alpha-2}}$, except when $\alpha = 2$. If $\alpha = 2$, then $f(t) = (2\kappa - \mu)t^2$. Since $f(t)$ is not constant, $\mu \neq 2\kappa$, and thus $f(t)$ has no positive zeros.
    \end{itemize}
    
    In all subcases, the positive zeros of $f(t)$ are bounded away from $0$ and $\infty$. Since they are isolated, there are only finitely many. Therefore, $\mathcal{S}$ is finite.

    Next, recall that $\tilde{\mathcal{S}} = \mathcal{S} \cap \big(\{t \in [0,+\infty) : g''(t) = 0\}\cup \{\infty:\tilde{g}''(0) = 0 \}\big)$. By Propositions \ref{prop g(t)} and \ref{prop g(0)}, we deduce that $\tilde{\mathcal{S}}$ is a singleton. Specifically:
    \begin{itemize}
        \item If there exists $t_0 \in \tilde{\mathcal{S}}$ with $0 < t_0 < \infty$, then Proposition \ref{prop g(t)} implies that $\alpha, \beta \in (1,2)$ and $t_0 = \sqrt{\frac{2-\beta}{2-\alpha}}$. Moreover, Proposition \ref{prop g(0)} shows that $0, \infty \notin \tilde{\mathcal{S}}$, so $\tilde{\mathcal{S}} = \{t_0\}$.
        \item If no such $t_0$ exists, then $\tilde{\mathcal{S}} \subset \{0, \infty\}$. By Proposition \ref{prop g(0)},
        \[
            \begin{cases}
                0 \in \tilde{\mathcal{S}} \implies (\beta, \lambda) = (2, 2\kappa), \\
                \infty \in \tilde{\mathcal{S}} \implies (\alpha, \mu) = (2, 2\kappa).
            \end{cases}
        \]
        If $\tilde{\mathcal{S}} = \{0, \infty\}$, then $(\alpha, \beta, \lambda, \mu) = (2, 2, 2\kappa, 2\kappa)$, a contradiction. Hence, $\tilde{\mathcal{S}}$ is a singleton.
    \end{itemize}

    We now prove that $\mathcal{S}$ itself is a singleton or doubleton in each scenario.

    \noindent{
        \textbf{Case 1: $\tilde{\mathcal{S}} = \{t_0\}$ with $t_0 = \sqrt{\frac{2-\beta}{2-\alpha}}$.}
    }
    In this case, $\alpha, \beta \in (1,2)$. We claim that $\mathcal{S} = \{t_0\}$. First, note that $0, \infty \notin \mathcal{S}$: if $0 \in \mathcal{S}$, then $\beta \geq 2$, contradicting to $\beta \in (1,2)$. Similarly, $\infty \notin \mathcal{S}$.

    Now, suppose for contradiction that there exists $t_1 \in \mathcal{S} \setminus \{t_0\}$ with $0 < t_1 < \infty$. Since $g'(t_0) = g'(t_1) = 0$, we have
    \begin{equation}\label{eq:gprime}
        \lambda = \mu t^{p-2} + \kappa \beta t^{\beta-2} - \kappa \alpha t^\beta \quad \text{for } t = t_0, t_1.
    \end{equation}
    Define the function
    \[
        \varphi(t) := \mu t^{p-2} + \kappa \beta t^{\beta-2} - \kappa \alpha t^\beta.
    \]
    Then $\varphi(t_0) = \varphi(t_1) = \lambda$. We compute the derivative:
    \begin{align*}
        \varphi'(t) &= (p-2)\mu t^{p-3} + \kappa \beta (\beta-2) t^{\beta-3} - \kappa \alpha \beta t^{\beta-1} \\
        &= t^{\beta-3} \left[(p-2)\mu t^{\alpha} + \kappa \beta (\beta-2) - \kappa \alpha \beta t^2 \right].
    \end{align*}
    Substitute $\mu = \frac{2\kappa \beta}{p-2} t_0^{2-\alpha}$ (by Proposition \ref{prop g(t)}-(ii)):
    \[
        \varphi'(t) = \kappa \beta t^{\beta-3} \left[2 t_0^{2-\alpha} t^{\alpha} + (\beta-2) - \alpha t^2 \right] =: \kappa \beta t^{\beta-3} Q(t).
    \]
    Now, compute
    \[
        Q'(t) = 2\alpha t_0^{2-\alpha} t^{\alpha-1} - 2\alpha t = 2\alpha t \left[ (t_0/t)^{2-\alpha} - 1 \right].
    \]
    Since $\alpha < 2$, we have $Q'(t) > 0$ for $t\in(0,t_0)$; $Q'(t) < 0$ for $t\in (t_0,\infty)$. Thus, $Q(t)$ is increasing on $(0, t_0)$ and decreasing on $(t_0, \infty)$, with maximum at $t = t_0$. Note that
    \[
        Q(t_0) = 2 t_0^2 + (\beta-2) - \alpha t_0^2 = (2-\alpha)t_0^2 + (\beta-2) = 0,
    \]
    since $t_0^2 = \frac{2-\beta}{2-\alpha}$. Therefore, $Q(t) < 0$ for $t \neq t_0$, and hence $\varphi'(t) < 0$ for $0<t\neq t_0$.  This implies that $\varphi(t)$ decreases strictly in $(0,\infty)$, contradicting $\varphi(t_0) = \varphi(t_1)$. Thus, $\mathcal{S} = \{t_0\}$.

    \noindent \textbf{
        Case 2: $\tilde{\mathcal{S}} = \{0\}$ or $\tilde{\mathcal{S}} = \{\infty\}$.
    }
    By symmetry, we consider only $\tilde{\mathcal{S}} = \{0\}$. Then $(\beta, \lambda) = (2, 2\kappa)$, and either $\alpha > 2$ or $\alpha = 2$ with $\mu < 2\kappa$.
    We claim that there is no $t_1\in\mathcal{S}$ with $0<t_1<\infty$. If not, since $\tilde{\mathcal{S}}=\{0\}$, we have $g''(0) = 0$, but $g''(t_1) > 0$. The condition \eqref{g''>=0} gives
    \[
        \alpha(2-\alpha)\kappa t_1^\beta + \kappa t_1^{\beta-2} \alpha\beta > (p-2)\lambda.
    \]
    Substitute $\beta = 2$, $\lambda = 2\kappa$:
    \[
        \alpha(2-\alpha)\kappa t_1^2 + 2\kappa \alpha > 2\kappa \alpha.
    \]
    Simplifying, we get $\alpha(2-\alpha) t_1^2 > 0$, which implies $\alpha < 2$. This contradicts the assumption that $\alpha > 2$ or $\alpha = 2$ with $\mu < 2\kappa$.
   
    Hence $\mathcal{S} = \{0\}$ or $\{0,\infty\}$. More precisely,
    \begin{itemize}
        \item If $\beta=2\le \alpha$, $\mu<\lambda = 2\kappa$, then $\mathcal{S} = \{0\}$.
        \item If $\beta=2<\alpha$, $\mu=\lambda = 2\kappa$, then $S=\{0,\infty\}$.
    \end{itemize}

    The case $\tilde{\mathcal{S}} = \{\infty\}$ is analogous by symmetry.
\end{proof}

\subsection{Some Elementary Inequalities}
\begin{lemma}\label{lem 1}
    For any $m>0$, there exists $C_1$ such that for any $x,y\in\R$, there holds
    \[ |x+y|^\iota - |x|^\iota \le \iota |x|^{\iota-2}xy + \left(\frac{\iota(\iota-1)}{2}+m\right)|x|^{\iota-2}|y|^2 + C_1|y|^\iota,\quad \text{if}\,\, \iota\ge 2;\]
    \[ |x+y|^\iota - |x|^\iota \le \iota |x|^{\iota-2}xy + \left(\frac{\iota(\iota-1)}{2}+m\right)\frac{(|x|+C_1|y|)^\iota}{|x|^2+|y|^2}|y|^2,\quad \text{if}\,\, 1<\iota<2.\]
\end{lemma}
For the proof of it, see \cite[Lemma 2.4]{Figalli2022}.

\begin{lemma}\label{Lem2}
    For any $m>0$, there exists $C_2>0$ such that for any $x,y,z,w\in\R$ with $x,z>0$, there holds
    \begin{enumerate}[(i)]
        \item If $\alpha,\beta \ge 2$, then
        \begin{equation}\label{alpha beta>2}
            \begin{aligned}
                & |x+y|^\alpha|z+w|^\beta -x^\alpha z^\beta\\
                &\le \alpha x^{\alpha-1}z^\beta y + \beta x^\alpha z^{\beta -1}w \\
                &\quad + \left(\frac{\alpha(\alpha-1)}{2} + m\right)x^{\alpha-2}z^\beta y^2 + \left(\frac{\beta(\beta-1)}{2}+m\right)x^\alpha z^{\beta-2} w^2 + \alpha\beta x^{\alpha-1}z^{\beta-1}yw \\
                &\quad + C_2 \Big( x^\alpha|w|^\beta + z^\beta|y|^\alpha + x^{\alpha-1}z^{\beta-2}|y||w|^2 + x^{\alpha-2}z^{\beta-1}|y|^2|w| + |y|^\alpha|w|^\beta \Big).
            \end{aligned}
        \end{equation}
        In particular, if $\alpha=2$ and $\beta>2$, then
        \begin{equation}
            \begin{aligned}
                & |x+y|^2|z+w|^\beta -x^2 z^\beta \\
                &\le 2xz^\beta y + \beta x^2z^{\beta-1} w + z^\beta y^2 + \left(\frac{\beta(\beta-1)}{2}+m\right)x^2z^{\beta-2}w^2 + 2\beta xz^{\beta-1}yw \\
                &\quad + C_2 \Big( x^2|w|^\beta + xz^{\beta-2}|y||w|^2 + z^{\beta-1}|y|^2|w| + |y|^2|w|^\beta \Big).
            \end{aligned}
        \end{equation}

        If $\alpha=\beta =2$, then
        \begin{equation}\label{alpha=beta=2}
            \begin{aligned}
                & |x+y|^2|z+w|^2 - x^2z^2 \\
               & \le{}  2 xz^2y + 2x^2zw + z^2y^2 + x^2w^2 + 4xyzw  \\
                & + C_2(x|y||w|^2+z|y|^2|w| + |y|^2|w|^2).
            \end{aligned}
        \end{equation}
        \item If $1<\alpha<2\le \beta$, then
        \begin{equation}\label{alpha<2<beta}
            \begin{aligned}
                & |x+y|^\alpha|z+w|^\beta -x^\alpha z^\beta\\
                &\le \alpha x^{\alpha-1}z^\beta y + \beta x^\alpha z^{\beta -1}w \\
                &\quad + \left(\frac{\alpha(\alpha-1)}{2} + m\right)x^{\alpha-2}z^\beta y^2 + \left(\frac{\beta(\beta-1)}{2}+m\right)x^\alpha z^{\beta-2} w^2 + \alpha\beta x^{\alpha-1}z^{\beta-1}yw \\
                &\quad + C_2 \Big( x^\alpha|w|^\beta + x^{\alpha-1}z^{\beta-2}|y||w|^2 + x^{-1}{z}^{\beta-1} |y|^{\alpha+1}|w|+ x^{-1}z^{\beta}|y|^{\alpha+1} + |y|^\alpha|w|^\beta \Big).
            \end{aligned}
        \end{equation}
        In particular, if $\beta=2$, then
        \begin{equation}
            \begin{aligned}
                & |x+y|^\alpha|z+w|^2 - x^\alpha z^2 \\
                &\le \alpha x^{\alpha-1}z^2y + 2x^\alpha zw\\
                &\quad + \left(\frac{\alpha(\alpha-1)}{2}+m\right) x^{\alpha-2}z^2y^2 + x^\alpha w^2 + 2\alpha x^{\alpha-1}zyw \\
                &\quad + C_2 \Big( x^{\alpha-1}|y| |w|^2 + x^{-1}z^2 |y|^{\alpha+1} + x^{-1}z |y|^{\alpha+1}|w| + |y|^\alpha |w|^2\Big).
            \end{aligned}
        \end{equation}
        \item If $\alpha,\beta \in (1,2)$, then
        \begin{equation}\label{<<2}
            \begin{aligned}
                & |x+y|^\alpha|z+w|^\beta -x^\alpha z^\beta\\
                &\le \alpha x^{\alpha-1}z^\beta y + \beta x^\alpha z^{\beta -1}w \\
                &\quad + \left(\frac{\alpha(\alpha-1)}{2} + m\right)x^{\alpha-2}z^\beta y^2 + \left(\frac{\beta(\beta-1)}{2}+m\right)x^\alpha z^{\beta-2} w^2 + \alpha\beta x^{\alpha-1}z^{\beta-1}yw \\
                &\quad + C_2 \Big( x^\alpha z^{-1}|w|^{\beta+1} + x^{-1}z^\beta |y|^{\alpha + 1} \\
                & \phantom{\quad + C_2 \Big( } + x^{\alpha-1}z^{\frac{\beta-1}{2}}|y||w|^{\frac{\beta+1}{2}} + x^{\frac{\alpha-1}{2}}z^{\beta-1} |y|^{\frac{\alpha+1}{2}}|w| + |y|^\alpha|w|^\beta \Big).
            \end{aligned}
        \end{equation}
    \end{enumerate}
    By symmetry, similar inequalities holds for $\alpha\ge \beta$.
\end{lemma}
\begin{proof}
    Without loss of generality, let $x=z=1$. For \emph{(i)}, it is equal to show
    \begin{equation}
        \begin{aligned}
            &|1+y|^\alpha |1+w|^\beta \\
            & \le{}  1 + \alpha y + \beta w  \\
            & \quad  +  \left(\frac{\alpha(\alpha-1)}{2} + m\right) y^2 + \left(\frac{\beta(\beta-1)}{2}+m\right) w^2 + \alpha\beta yw \\
            & \quad  + C_2\Big( |y|^{\alpha} + |w|^{\beta} + |y||w|^2 + |y|^2|w| + |y|^\alpha |w|^\beta \Big).
        \end{aligned}
    \end{equation}
    By Lemma \ref{lem 1},
    \begin{align*}
        & |1+y|^\alpha|1+w|^\beta \\
        &\le \left[ 1 + \alpha y + \left( \frac{\alpha(\alpha-1)}{2}+m\right)|y|^2 + C_1|y|^\alpha \right]\\
        & \quad \cdot \left[ 1 + \beta w + \left( \frac{\beta(\beta-1)}{2}+m\right)|w|^2 + C_1|w|^\beta \right] \\
        &= 1 + \alpha y + \beta w + \left( \frac{\alpha(\alpha-1)}{2}+m\right)|y|^2 + \left( \frac{\beta(\beta-1)}{2}+m\right)|w|^2 + \alpha\beta yw \\
        &\quad + C\Big( |y|^\alpha + |w|^\beta + |y||w|^2 + |y|^2 |w| + |y|^\alpha|w|^\beta \\
        &\qquad + |y|^2|w|^2 + |y||w|^\beta + |y|^\alpha|w| + |y|^2|w|^\beta + |y|^\alpha|w|^2 \Big).
    \end{align*}
    Since $1+|y|^\alpha \gtrsim |y|^t$ for any  {$t\in (0,\alpha]$}, we can absorb $|y|^2|w|^\beta$ and $|y||w|^\beta$ by $|y|^\alpha|w|^\beta + |w|^\beta$. Similarly, $|y|^\alpha|w|$ and $|y|^\alpha|w|^2$ can be absorbed by $|y|^\alpha + |y|^\alpha|w|^\beta$. Lastly, $|y|^2|w|^2$ is absorbed by $|y||w|^2 + |y|^\alpha|w|^2$, and hence by $|y||w|^2 + |y|^\alpha|w|^\beta + |y|^\alpha$ for the same reason.

    If $\alpha=2$, $\beta>2$, then
    \begin{align*}
        & |1+y|^\alpha|1+w|^\beta \\
        &\le (1+2y+|y|^2)\cdot \left[ 1 + \beta w + \left( \frac{\beta(\beta-1)}{2}+m\right)|w|^2 + C_1|w|^\beta \right] \\
        &= 1 + 2y + \beta w + |y|^2 + \left( \frac{\beta(\beta-1)}{2}+m\right)|w|^2 + 2\beta yw \\
        &\quad + C\Big( |y|^2|w| + |y||w|^2 + |y|^2|w|^2 + |w|^\beta + |y||w|^\beta + |y|^2|w|^\beta \Big).
    \end{align*}
    Since $\beta>2$, we can absorb $|y||w|^\beta$ by $|y|^2|w|^\beta +|w|^\beta$ and $|y|^2|w|^2$ by $|y|^2|w| + |y|^2|w|^\beta$. The case $\alpha=\beta=2$ is trivial and we omit the proof.

    For \emph{(ii)}, it is equal to show
    \begin{equation}
        \begin{aligned}
            & |1+y|^\alpha|1+w|^\beta \\
            &\le 1 + \alpha y + \beta w + \left( \frac{\alpha(\alpha-1)}{2}+m\right)|y|^2 + \left( \frac{\beta(\beta-1)}{2}+m\right)|w|^2 + \alpha\beta yw \\
            &\quad + C_2 \Big( |w|^{\beta} + |y||w|^2 + |y|^{\alpha+1}|w| + |y|^{\alpha+1} + |y|^\alpha|w|^\beta \Big).
        \end{aligned}
    \end{equation}
    By Lemma \ref{lem 1}, fix a $m'>0$ to be determined,
    \begin{align*}
        & |1+y|^\alpha|1+w|^\beta \\
        &\le \left[ 1 + \alpha y + \left( \frac{\alpha(\alpha-1)}{2}+m'\right)\frac{(1+C_1|y|)^\alpha }{1+|y|^2} |y|^2\right]\cdot \\
        & \quad \left[ 1 + \beta w + \left( \frac{\beta(\beta-1)}{2}+m'\right)|w|^2 + C_1|w|^\beta \right] \\
        &= 1 + \alpha y + \beta w + \left( \frac{\alpha(\alpha-1)}{2}+m'\right)\frac{(1+C_1|y|)^\alpha }{1+|y|^2}|y|^2 \\
        & \quad+ \left( \frac{\beta(\beta-1)}{2}+m'\right)|w|^2 + \alpha\beta yw \\
        & \quad+ C\Big( |w|^\beta + |y||w|^2 +|y||w|^\beta  + \frac{(1+C_1|y|)^\alpha }{1+|y|^2}|y|^2\big(|w| + |w|^2 + |w|^\beta\big) \Big).
    \end{align*}
    Since $\alpha>1$, for any $\epsilon>0$, there exists $D(\epsilon,C_1)>0$ such that
    \begin{align*}
        (1+C_1|y|)^\alpha \le{}& (1+\epsilon) + D|y|^\alpha \\
        \le{}& (1+|y|^2)\big( (1+\epsilon) + \frac 12 D|y|^{\alpha-1}\big).
    \end{align*}
    Thus,
    \[
        \left( \frac{\alpha(\alpha-1)}{2}+m'\right)\frac{(1+C_1|y|)^\alpha }{1+|y|^2} |y|^2 \le  \left( \frac{\alpha(\alpha-1)}{2}+m'\right)(1+\epsilon) |y|^2 + C |y|^{\alpha+1},
    \]
    \[ 
        \frac{(1+C_1|y|)^\alpha }{1+|y|^2} |y|^2|w| \le C (|y|^2|w| + |y|^{\alpha+1}|w|).
    \]
    \[ 
        \frac{(1+C_1|y|)^\alpha }{1+|y|^2} |y|^2|w|^\beta < (1+C_1|y|)^\alpha|w|^\beta \lesssim |w|^\beta + |y|^\alpha|w|^\beta.
    \]
    By $\beta\ge 2$, $|w|^2 \lesssim |w|+|w|^\beta$.
    Let $m'>0$ be that $\frac{\alpha(\alpha-1)}{2}+m = \left( \frac{\alpha(\alpha-1)}{2}+m'\right)(1+\epsilon)$, then
    \begin{align*}
        & |1+y|^\alpha|1+w|^\beta\\
        &\le 1 + \alpha y + \beta w + \left( \frac{\alpha(\alpha-1)}{2}+m\right)|y|^2 + \left( \frac{\beta(\beta-1)}{2}+m\right)|w|^2 + \alpha\beta yw \\
        &\quad + C \Big( |y|^{\alpha+1} + |w|^{\beta} + |y||w|^2 +  |y||w|^\beta + |y|^2|w| + |y|^{\alpha+1}|w| + |y|^\alpha|w|^\beta\Big).
    \end{align*}
    Since $|y||w|^\beta \lesssim |w|^\beta + |y|^\alpha|w|^\beta$ and
    \begin{align*}
        |y|^2|w| ={}& (|y||w|^2)^{1-1/\alpha} (|y|^{\alpha+1})^{1-1/\alpha} (|y|^{\alpha+1}|w|)^{2/\alpha-1} \\
        \lesssim{}& |y||w|^2 + |y|^{\alpha+1} + |y|^{\alpha+1}|w|,
    \end{align*}
    we are done. If $\beta=2$, then
    \begin{align*}
        & |1+y|^\alpha|1+w|^2 \\
        &\le \left[ 1 + \alpha y + \left( \frac{\alpha(\alpha-1)}{2}+m'\right)\frac{(1+C_1|y|)^\alpha }{1+|y|^2} |y|^2\right]\cdot (1 + 2w + |w|^2) \\
        &= 1 + \alpha y + 2w + \left( \frac{\alpha(\alpha-1)}{2}+m'\right)\frac{(1+C_1|y|)^\alpha }{1+|y|^2}|y|^2 + |w|^2 + 2\alpha yw \\
        &\quad + C\Big( |y||w|^2 + \frac{(1+C_1|y|)^\alpha }{1+|y|^2}|y|^2(|w| + |w|^2) \Big)\\
        &\le  1 + \alpha y + 2w + \left( \frac{\alpha(\alpha-1)}{2}+m\right)|y|^2 + |w|^2 + 2\alpha yw \\
        &\quad + C\Big( |y|^{\alpha+1} + |y||w|^2 + |y|^{\alpha+1}|w| + |y|^{\alpha}|w|^2\Big).
    \end{align*}

    Finally, for \emph{(iii)}, it equals to show
    \begin{equation}\label{elementary 1}
        \begin{aligned}
            &|1+y|^\alpha |1+w|^\beta\\
            & \le{}  1 + \alpha y + \beta w  \\
            & \quad  +  \left(\frac{\alpha(\alpha-1)}{2} + m\right) y^2 + \left(\frac{\beta(\beta-1)}{2}+m\right) w^2 + \alpha\beta yw \\
            & \quad  + C_2\Big( |y|^{\alpha+1} + |w|^{\beta+1} + |y||w|^{\frac{\beta+1}{2}} + |y|^{\frac{\alpha+1}{2}}|w| + |y|^\alpha |w|^\beta \Big) .
        \end{aligned}
    \end{equation}
    By Lemma \ref{lem 1} and the estimates in the proof of \emph{(ii)},
    \begin{align*}
        & |1+y|^\alpha|1+w|^\beta \\
        &\le \left[ 1 + \alpha y + \left( \frac{\alpha(\alpha-1)}{2}+m'\right)\frac{(1+C_1|y|)^\alpha }{1+|y|^2}|y|^2 \right]\cdot \\
        & \quad \left[ 1 + \beta w + \left( \frac{\beta(\beta-1)}{2}+m'\right)\frac{(1+C_1|w|)^\beta }{1+|w|^2}|w|^2\right] \\
        &= 1 + \alpha y + \beta w + \left( \frac{\alpha(\alpha-1)}{2}+m'\right)\frac{(1+C_1|y|)^\alpha }{1+|y|^2}|y|^2 \\
        & \quad + \left( \frac{\beta(\beta-1)}{2}+m'\right)\frac{(1+C_1|w|)^\beta }{1+|w|^2}|w|^2 + \alpha\beta yw \\
        & \quad + C\Big( \frac{(1+C_1|y|)^\alpha }{1+|y|^2}|y|^2|w| +\frac{(1+C_1|w|)^\beta }{1+|w|^2}|y||w|^2 \\
        & \quad\quad  + \frac{(1+C_1|y|)^\alpha }{1+|y|^2}\frac{(1+C_1|w|)^\beta }{1+|w|^2}|y|^2|w|^2 \Big) \\
        &\le 1 + \alpha y + \beta w + \left( \frac{\alpha(\alpha-1)}{2}+m\right)|y|^2 + \left( \frac{\beta(\beta-1)}{2}+m\right)|w|^2 + \alpha\beta yw \\
        & \quad + C \bigg( |y|^{\alpha+1} + |w|^{\beta+1} + \frac{(1+C_1|y|)^\alpha }{1+|y|^2}|y|^2|w| + \frac{(1+C_1|w|)^\beta }{1+|w|^2}|y||w|^2\\
        & \qquad + \frac{(1+C_1|y|)^\alpha }{1+|y|^2}\frac{(1+C_1|w|)^\beta }{1+|w|^2}|y|^2|w|^2 \bigg).
    \end{align*}
    Since
    \[
        (1+C_1|y|)^\alpha \lesssim (1+|y|^2)(|y|^{\frac{\alpha-1}{2}} + |y|^{\alpha-2}),
    \]
    we obtain
    \[
        \frac{(1+C_1|y|)^\alpha }{1+|y|^2}|y|^2|w| \lesssim |y|^{\frac{\alpha+1}{2}}|w| + |y|^\alpha|w|.
    \]
    Similarly,
    \[
        \frac{(1+C_1|w|)^\beta }{1+|w|^2}|y||w|^2 \lesssim |y||w|^{\frac{\beta+1}{2}} + |y||w|^{\beta}.
    \]
    \[ 
        \frac{(1+C_1|y|)^\alpha }{1+|y|^2}\frac{(1+C_1|w|)^\beta }{1+|w|^2}|y|^2|w|^2 \lesssim (|y|^{\frac{\alpha+1}2} + |y|^\alpha)(|w|^{\frac{\beta+1}{2}} + |w|^\beta).
    \]
    It is easy to show that $|y|^{\frac{\alpha+1}{2}}|w|^{\frac{\beta+1}{2}}$, $|y|^{\frac{\alpha+1}{2}}|w|^\beta$, $|y|^\alpha|w|^{\frac{\beta+1}{2}}$ can be absorbed by
    \[
        |y|^{\alpha+1} + |w|^{\beta+1} + |y|^{\frac{\alpha+1}{2}}|w| + |y||w|^{\frac{\beta+1}{2}} + |y|^\alpha|w|^\beta,
    \]
    since $(\frac{\alpha+1}2, \frac{\beta+1}{2}), (\alpha,\frac{\beta+1}{2}), (\frac{\alpha+1}2,\beta)$ are in the closure of the convex set formed by the five points $(\alpha+1,0),(0,\beta+1),(\alpha,\beta),(\frac{\alpha+1}2,1), (1,\frac{\beta+1}2)$ (See Figure \ref{figure}).
\end{proof}

\begin{figure}[h]
    \centering
    \begin{tikzpicture}

        \def\x{1.7}
        \def\y{1.6}
        \def\z{2}

        \coordinate (A) at (\z*\x+\z,0);
        \coordinate (B) at (\z*\x,\z*\y);
        \coordinate (C) at (0,\z*\y+\z);
        \coordinate (D) at (\z,\y*0.5*\z+0.5*\z);
        \coordinate (E) at (\x*0.5*\z + 0.5*\z,\z);
        \coordinate (F) at (\x*0.5*\z+0.5*\z,\y*\z);
        \coordinate (G) at (\x*\z,\y*0.5*\z+0.5*\z);
        \coordinate (H) at (\x*0.5*\z+0.5*\z,\y*0.5*\z+0.5*\z);
        \draw[thin,dotted] (0,0) grid (\x*\z+\z+1,\y*\z+\z+1);
        \draw[very thick, blue] (A) -- (B) -- (C) -- (D) -- (E) -- cycle;
        \draw[->] (0,0) -- (0,\y*\z + \z+1);
        \draw[->] (0,0) -- (\x*\z + \z+1,0);
        \fill (F) circle [radius = 0.05];
        \fill (G) circle [radius = 0.05];
        \fill (H) circle [radius = 0.05];
        \fill[blue] (A) circle [radius = 0.05];
        \fill[blue] (B) circle [radius = 0.05];
        \fill[blue] (C) circle [radius = 0.05];
        \fill[blue] (D) circle [radius = 0.05];
        \fill[blue] (E) circle [radius = 0.05];
        \node[below] at (H) {B};
        \node[below] at (G) {C};
        \node[left] at (F) {A};
        \node[below] at (A) {$(\alpha+1,0)$};
        \node[right] at (B) {$(\alpha,\beta)$};
        \node[left] at (C) {$(0,\beta+1)$};
        \node[left] at (D) {$(1,\frac{\beta+1}{2})$};
        \node[left] at (E) {$(\frac{\alpha+1}{2},1)$};
        \node[below] at (0,0) {$O$};

        \node[anchor=north east] at (current bounding box.north east) {
        \begin{tabular}{r@{\ }l}
            A: & ($\frac{\alpha+1}2$, $\beta$)\\[-5pt]
            B: & ($\frac{\alpha+1}2$, $\frac{\beta+1}{2}$) \\[-5pt]
            C: & ($\alpha$, $\frac{\beta+1}{2}$)
        \end{tabular}
        };
    \end{tikzpicture}
    \caption{Lemma \ref{Lem2} case (iii)}\label{figure}
\end{figure}

\section{Qualitative Analysis and Existence of Minimizers} \label{sec3}

Fix a minimizer $t_0\in\mathcal{S}\backslash\{\infty\}$, for nonnegative $(u,v)\in\mathscr{D}\backslash {\{(0,0)\}}$, consider the problem
\begin{equation}\label{min pro}
    \inf_{\sigma\in\R,\, 0<w\in\mathcal{M},\,\|w\|_{(s)}=1} {\|\nabla(u-\sigma w)\|_2^2+\|\nabla(v-t_0 \sigma w)\|_2^2},
\end{equation}
where
\begin{equation}\label{s-norm}
        \|w\|_{(s)}:=\left(\int_{\R^N}\frac{|w|^{2^*(s)}}{|x|^s}\,\mathrm{d}x\right)^{\frac 1{2^*(s)}}.
    \end{equation}
It is easy to see there exists minimizer $(\sigma ,w)$. They satisfy
\begin{equation}\label{perp}
    \begin{aligned}
        & \int \nabla(u+t_0v-(1+t_0^2)\sigma w)\cdot \nabla w\,\mathrm{d}x = 0,\\
        & \int \nabla(u+t_0v-(1+t_0^2)\sigma w)\cdot \nabla(\partial_\tau w)\,\mathrm{d}x=0.
    \end{aligned}
\end{equation}
Here for $w=U(k_0,\tau_0)\in\mathcal{M}$,
\[
    \partial_\tau w := \frac{\mathrm{d}}{\mathrm{d}\tau}\Big|_{\tau=\tau_0} U(k_0,\tau).
\]
If $t_0=\infty$, consider instead
\begin{equation}\label{min infty}
    \inf_{\sigma \in\R,\, 0<w\in \mathcal{M},\, \|w\|_{(s)}=1} \|\nabla(v-\sigma w)\|_2^2,
\end{equation}
and the minimizer $(\sigma,w)$ satisfies
\begin{equation}\label{perp_inf}
    \begin{aligned}
        & \int \nabla(v-\sigma w)\cdot \nabla w\,\mathrm{d}x = 0,\\
        & \int \nabla(v-\sigma w)\cdot \nabla(\partial_\tau w)\,\mathrm{d}x=0.
    \end{aligned}
\end{equation}

We first derive another expression of $\sigma $ {when $t_0 \in\mathcal{S}\backslash\{\infty\}$}:
\begin{lemma}\label{k sup}
    Let $(\sigma,w)$ minimize the problem \eqref{min pro}  {with $w>0$}, then
    \begin{equation}
        \sigma  = \sup_{ {0<}w_1\in\mathcal{M},\,\|w_1\|_{(s)}=1} \frac{\int \nabla(u+t_0v)\cdot\nabla w_1\,\mathrm{d}x}{(1+t_0^2)\mu_s}\ge 0.
    \end{equation}
    As a result, for $t_0\in(0,\infty)$, $\sigma=0$ if and only if $u=v=0$; for $t_0=0$, $\sigma =0$ if and only if $u=0$.
\end{lemma}
\begin{proof}
    Since $(\sigma,w)$ minimize the problem \eqref{min pro}, for any other $(\sigma',w')$ with $ 0<w'\in \mathcal{M}, \|w'\|_{(s)}=1$ and $\sigma'\in\R$, we have
    \[
        \|\nabla(u+t_0v-(1+t_0^2)\sigma w)\|_2^2 \le \|\nabla(u+t_0v-(1+t_0^2)\sigma'w')\|_2^2.
    \]
    Let
    \[
        \sigma' = \frac{\int \nabla(u+t_0v)\cdot\nabla w'\,\mathrm{d}x}{(1+t_0^2)\mu_s},
    \]
    then the above inequality becomes
    \[
        \|\nabla(u+t_0v)\|_2^2 - (1+t_0^2)^2\sigma^2\mu_s \le \|\nabla(u+t_0v)\|_2^2 - (1+t_0^2)^2(\sigma')^2\mu_s,
    \]
    i.e. $\sigma^2\ge (\sigma')^2$. Notice that
    \[
        0 = \int \nabla(u+t_0v-(1+t_0^2)\sigma w)\cdot \nabla w\,\mathrm{d}x \implies \sigma = \frac{\int \nabla(u+t_0v)\cdot\nabla w\,\mathrm{d}x}{(1+t_0^2)\mu_s}.
    \]
    By $0<w\in\mathcal{M}, t_0\geq 0$ and $u, v$ are nonnegative,
    \begin{align*}
        \sigma=\frac{\int \nabla(u+t_0v)\cdot\nabla w\,\mathrm{d}x}{(1+t_0^2)\mu_s} ={}& \frac{1}{1+t_0^2}\int \frac{1}{|x|^s}w^{p-1}(u+t_0v)\ge 0.
    \end{align*}
    Similarly, $\sigma'\ge 0$, and so $\sigma \ge \sigma'$. By the arbitrary selection of $w'$, it is known that $\sigma$ attains the supreme \eqref{k sup}, and $\sigma=0$ if and only if $u+t_0v=0$.
\end{proof}
\begin{remark}\label{k>0}
    If $t_0=\infty$ and $(\sigma,w)$ minimize \eqref{min infty}, then
    \[
        \sigma = \sup_{0<w_1\in\mathcal{M},\,\|w_1\|_{(s)}=1}\frac{\int \nabla v\cdot\nabla w_1\,\mathrm{d}x}{\mu_s} {=\frac{1}{\mu_s}\int \frac{1}{|x|^s} w^p v\, \mathrm{d}x}\geq 0,
    \]
    and $\sigma=0$ if and only if $v=0$.
\end{remark}

By Lemma \ref{k sup} and Remark \ref{k>0}, we see that $\sigma=\sigma(u,v,t_0)$  which is determined by $u, v$ and $t_0$. Without causing any misunderstanding, we can omit the symbols $u,v$. And keep in mind that $\sigma(t_0)$ is a functional of $(u,v)$. Now we can state the qualitative stability result:
\begin{theorem}\label{lem 222}
    Assume $(\alpha, \beta, \lambda, \mu) \neq (2, 2, 2\kappa, 2\kappa)$. For any $\epsilon_1 > 0$, there exists $\epsilon_0 > 0$ such that if $\delta(u, v) < \epsilon_0 (\|\nabla u\|_2^2 + \|\nabla v\|_2^2)$, then for some minimizer $t_0 \in \mathcal{S}$ and the corresponding $\sigma = \sigma(t_0)$, the following estimate holds:
    \begin{equation}\label{eta small when delta small}
        \eta < \epsilon_1 \sigma^2,
    \end{equation}
    where for $t_0 \in [0, \infty)$,
    \[
    \eta = \eta(t_0) := \|\nabla (u - \sigma w)\|_2^2 + \|\nabla (v - t_0 \sigma w)\|_2^2,
    \]
    with $(\sigma, w)$  {being a minimizer of} problem \eqref{min pro}; and for $t_0 = \infty$,
    \[
    \eta = \eta(\infty) := \|\nabla u\|_2^2 + \|\nabla (v - \sigma w)\|_2^2,
    \]
    with $(\sigma, w)$  {being a minimizer of} problem \eqref{min infty}.
\end{theorem}

\begin{proof}
    We proceed by contradiction. Suppose the statement is false. Then there exists $\epsilon_1 > 0$ and a sequence $(u_i, v_i)$ with $\|\nabla u_i\|_2^2 + \|\nabla v_i\|_2^2 = 1$ such that
    \begin{align}
        & \delta(u_i, v_i) \le i^{-1}, \label{hypo1} \\
        & \eta_i(\tilde{t}) \ge \epsilon_1 \sigma_i(\tilde{t})^2 \quad \text{for all } \tilde{t} \in \mathcal{S}. \label{hypo2}
    \end{align}
    Here, for simplicity, we denote $\eta(u_i, v_i, \tilde{t})$ and $\sigma(u_i, v_i, \tilde{t})$ by $\eta_i(\tilde{t})$ and $\sigma_i(\tilde{t})$, respectively.

    Since $g(\tilde{t}) = \inf_{t > 0} g(t)$, we have
    \[
        \lambda u_i^p + \mu v_i^p + \kappa p u_i^\alpha v_i^\beta  {=\left( \frac{u_i^2 + v_i^2}{g(\frac{v_i}{u_i})} \right)^{p/2}} \le \left( \frac{u_i^2 + v_i^2}{g(\tilde{t})} \right)^{p/2}.
    \]
    Consequently,
    \begin{align*}
        \delta(u_i, v_i) &= \|\nabla u_i\|_2^2 + \|\nabla v_i\|_2^2 - g(\tilde{t}) \mu_s \left( \int \frac{1}{|x|^s} \left( \lambda u_i^p + \mu v_i^p + \kappa p u_i^\alpha v_i^\beta \right) \, \mathrm{d}x \right)^{2/p} \\
        &\ge \|\nabla u_i\|_2^2 + \|\nabla v_i\|_2^2 - \mu_s \left( \int \frac{1}{|x|^s} (u_i^2 + v_i^2)^{p/2} \, \mathrm{d}x \right)^{2/p} \\
        &\ge \|\nabla u_i\|_2^2 + \|\nabla v_i\|_2^2 - \mu_s \left( \|u_i\|_{(s)}^2 + \|v_i\|_{(s)}^2 \right) \\
        &= \delta(u_i) + \delta(v_i).
    \end{align*}
    The last inequality follows from the elementary estimate: for $p > 2$ and $f, g \in L^p(\R^N)$,
    \[
        \| \sqrt{f^2 + g^2} \|_p^2 = \| f^2 + g^2 \|_{p/2} \le \| f^2 \|_{p/2} + \| g^2 \|_{p/2} = \| f \|_p^2 + \| g \|_p^2.
    \]
    By \eqref{hypo1}, $\delta(u_i), \delta(v_i) \to 0$. Applying the stability result of the Hardy-Sobolev inequality (Theorem~\ref{HS-single}), there exist sequences $\{a_i f_i\}, \{b_i g_i\} \subset \mathcal{M}$ with $a_i, b_i \ge 0$, $f_i > 0$, $g_i > 0$, $\| f_i \|_{(s)} = \| g_i \|_{(s)} = 1$, such that
    \begin{align}
        \| \nabla (u_i - a_i f_i) \|_2 \to 0, \quad \| \nabla (v_i - b_i g_i) \|_2 \to 0. \label{limfg}
    \end{align}
    In particular, $\| \nabla u_i \|_2 \to \| \nabla (a_i f_i) \|_2$, $\| \nabla v_i \|_2 \to \| \nabla (b_i g_i) \|_2$, and
    \begin{equation} \label{limab1}
        \lim_{i \to \infty} (a_i^2 + b_i^2) = \mu_s^{-1},
    \end{equation}
    which implies that $\{a_i\}$ and $\{b_i\}$ are bounded. Thus, up to a subsequence, we may assume
    \begin{equation} \label{convergeab}
        a_i \to a_0, \quad b_i \to b_0 \quad \text{in } \R_{\ge 0}, \quad \text{with} \quad a_0^2 + b_0^2 = \mu_s^{-1}.
    \end{equation}
    By symmetry, assume $a_0 \neq 0$ and thus $a_0\in (0,\infty)$. Write $f_i = U(1, \lambda_i)$, $g_i = U(1, \mu_i)$, and let $\gamma_i = \mu_i / \lambda_i$. Then,
    \begin{align}
        \int \frac{1}{|x|^s} f_i^\alpha g_i^\beta ={}& k_0^p \int \frac{1}{|x|^s} (1+|\lambda_i x|^{2-s})^{-\frac{N-2}{2-s}\alpha}(1+|\mu_ix|^{2-s})^{-\frac{N-2}{2-s}\beta}\lambda_i^{\frac{N-2}{2}\alpha}\mu_i^{\frac{N-2}{2}\beta}\,\mathrm{d}x \notag \\
        ={}& k_0^p  {\int \frac{1}{|x|^s}U(1,1)^\alpha U(1,\gamma_i)^\beta \, \mathrm{d}x}\label{fg1} \\
        ={}& k_0^p {\int \frac{1}{|x|^s}U(1,\tfrac{1}{\gamma_i})^\alpha U(1,1)^\beta \, \mathrm{d}x}\label{fg2}
    \end{align}

    \noindent{\textbf{Case 1: $b_0 > 0$.}}

    If $\{\gamma_i\}$ is unbounded, up to a subsequence, we have that $\frac 1{\gamma_i} \to 0$. Then $U(1, \tfrac{1}{\gamma_i}) \rightharpoonup 0$ in $D_0^{1,2}(\R^N)$, and thus $U(1, \tfrac{1}{\gamma_i})^\alpha \rightharpoonup 0$ in $L^{p/\alpha}(\R^N,\frac{1}{|x|^s}\,\mathrm{d}x)$. Since $U(1,1)^\beta \in L^{p/\beta}(\R^N,\frac{1}{|x|^s}\,\mathrm{d}x)$ (the dual space of $L^{p/\alpha}(\R^N,\frac{1}{|x|^s}\,\mathrm{d}x)$ by $p = \alpha + \beta$), we have
    \begin{equation} \label{convzero}
        \int \frac{1}{|x|^s} U(1, \tfrac{1}{\gamma_i})^\alpha U(1,1)^\beta \, \mathrm{d}x \to 0.
    \end{equation}
    Hence, by \eqref{fg2}, $\int \frac{1}{|x|^s} f_i^\alpha g_i^\beta \, \mathrm{d}x \to 0$. Similarly, if $\gamma_i \to 0$, using \eqref{fg1}, the same conclusion holds. By \eqref{hypo1},
    \begin{equation}\label{eq:20251111-1240}
    \begin{aligned}
        1 - i^{-1} &\le g(\tilde{t}) \mu_s \left( \int \frac{1}{|x|^s} \left( \lambda u_i^p + \mu v_i^p + \kappa p u_i^\alpha v_i^\beta \right) \,\mathrm{d}x \right)^{2/p} \\
        &= g(\tilde{t}) \mu_s \left( \int \frac{1}{|x|^s} \left( \lambda (a_i f_i)^p + \mu (b_i g_i)^p + \kappa p (a_i f_i)^\alpha (b_i g_i)^\beta \right) \,\mathrm{d}x \right)^{2/p} + o(1) \\
        &= g(\tilde{t}) \mu_s \left( \lambda a_i^p + \mu b_i^p + \kappa p a_i^\alpha b_i^\beta \int \frac{1}{|x|^s} f_i^\alpha g_i^\beta \,\mathrm{d}x \right)^{2/p} + o(1).
    \end{aligned}
    \end{equation}
    So, for either $\gamma_i\rightarrow \infty$ or $\gamma_i\rightarrow 0$, by letting $i\rightarrow \infty$ in \eqref{eq:20251111-1240}, we have that
    \begin{align*}
        1\leq{}&g(\tilde{t}) \mu_s \left( \lambda a_0^p + \mu b_0^p\right)^{2/p}\\
        ={}& g(\tilde{t}) \frac{1}{a_0^2+b_0^2}\left( \lambda a_0^p + \mu b_0^p\right)^{2/p}\quad\hbox{(by \eqref{convergeab})}\\
        <{}&g(\tilde{t})\frac{1}{a_0^2+b_0^2}\left( \lambda a_0^p + \mu b_0^p+\kappa p a_0^\alpha b_0^\beta\right)^{2/p} \quad\hbox{(since $a_0>0, b_0>0$)}\\
        ={}&\frac{g(\tilde{t})}{g(b_0/a_0)}\leq 1~\hbox{(since $\tilde{t}\in \mathcal{S}$)},
    \end{align*}
    which is a contradiction.

    Thus, we may assume $\gamma_i \to \gamma_0 \in (0, \infty)$.
    Letting $i \to \infty$ again in \eqref{eq:20251111-1240} and using H\"older's inequality and \eqref{limab1}, we obtain
    \begin{align*}
        1 &\le g(\tilde{t}) \mu_s \left( \lambda a_0^p + \mu b_0^p + \kappa p a_0^\alpha b_0^\beta \int \frac{1}{|x|^s} U(1,1)^\alpha U(1, \gamma_0)^\beta \,\mathrm{d}x \right)^{2/p} \\
        &\le g(\tilde{t}) \mu_s \left( \lambda a_0^p + \mu b_0^p + \kappa p a_0^\alpha b_0^\beta \right)^{2/p} \\
        &= g(\tilde{t}) \frac{1}{a_0^2+b_0^2} \left( \lambda a_0^p + \mu b_0^p + \kappa p a_0^\alpha b_0^\beta \right)^{2/p} \\
        &= \frac{g(\tilde{t})}{g(b_0/a_0)} \le 1.
    \end{align*}
    Hence, equality holds throughout, and
    \begin{equation} \label{limab2}
        g(b_0 / a_0) = g(\tilde{t}), \quad \gamma_0 = 1.
    \end{equation}
    Thus, there exists $t_0 \in \mathcal{S}$ such that $b_0 / a_0 = t_0$. Since $a_0 \in (0,\infty)$, we have $t_0 \in (0, \infty)$, and $(1+t_0^2)a_0^2=a_0^2+b_0^2=\mu_{s}^{-1}$, which gives that
    \begin{equation} \label{a0value}
        a_0 = \left( \frac{1}{(1 + t_0^2) \mu_s} \right)^{1/2}.
    \end{equation}
    
    Now, consider \eqref{hypo2}. We compute
    \begin{align*}
        \eta_i(t_0) &= \| \nabla u_i \|_2^2 + \| \nabla (\sigma_i w) \|_2^2 + \| \nabla v_i \|_2^2 + \| \nabla (t_0 \sigma_i w) \|_2^2 \\
        &\quad - 2 \int \left( \nabla u_i \cdot \nabla (\sigma_i w) + \nabla v_i \cdot \nabla (t_0 \sigma_i w) \right) \,\mathrm{d}x \\
        &= \| \nabla u_i \|_2^2 + \| \nabla v_i \|_2^2 + (1 + t_0^2) \sigma_i^2 \mu_s - 2 \int \nabla (u_i + t_0 v_i) \cdot \nabla (\sigma_i w) \,\mathrm{d}x \\
        &= \| \nabla u_i \|_2^2 + \| \nabla v_i \|_2^2 - (1 + t_0^2) \sigma_i(t_0)^2 \mu_s = 1 - (1 + t_0^2) \sigma_i(t_0)^2 \mu_s.
    \end{align*}
    Then \eqref{hypo2} implies
    \begin{equation} \label{hypo2ver2}
        \sigma_i(t_0)^2 \le \frac{1}{(1 + t_0^2) \mu_s + \epsilon_1}.
    \end{equation}
    Since $\gamma_0 = 1$, we have $U(1, \gamma_i) \rightharpoonup U(1,1)$ in $D_0^{1,2}(\R^N)$, and
    \begin{equation} \label{convint}
        \int \frac{1}{|x|^s} U(1,1)^{p-1} U(1, \gamma_i) \,\mathrm{d}x \to \int \frac{1}{|x|^s} U(1,1)^p \,\mathrm{d}x = 1.
    \end{equation}
    Thus,
    \begin{align*}
        \int \nabla (f_i - g_i) \cdot \nabla f_i \,\mathrm{d}x &= \mu_s \int \frac{1}{|x|^s} f_i^{p-1} (f_i - g_i) \,\mathrm{d}x \\
        &= \mu_s - \mu_s \int \frac{1}{|x|^s} U(1, \lambda_i)^{p-1} U(1, \mu_i) \,\mathrm{d}x \\
        &= \mu_s - \mu_s \int \frac{1}{|x|^s} U(1,1)^{p-1} U(1, \gamma_i) \,\mathrm{d}x \to 0.
    \end{align*}
    By Lemma \ref{k sup},
    \begin{align*}
        \sigma_i(t_0) &\ge \frac{\int \nabla (u_i + t_0 v_i) \cdot \nabla f_i \,\mathrm{d}x}{(1 + t_0^2) \mu_s} \\
        &= \frac{\int \nabla (u_i - a_i f_i + t_0 (v_i - b_i g_i)) \cdot \nabla f_i \,\mathrm{d}x + t_0 b_i \int \nabla (g_i - f_i) \cdot \nabla f_i \,\mathrm{d}x}{(1 + t_0^2) \mu_s} \\
        &\quad + \frac{(a_i + t_0 b_i) \mu_s}{(1 + t_0^2) \mu_s}.
    \end{align*}
    Letting $i \to \infty$ in \eqref{hypo2ver2}, and using \eqref{limfg}, \eqref{limab2}, \eqref{a0value}, and the convergence above, we obtain
    \[
        \left( \frac{1}{(1 + t_0^2) \mu_s + \epsilon_1} \right)^{1/2} \ge \limsup_{i\to \infty }\sigma_i(t_0) \ge \frac{(a_0 + t_0 b_0) \mu_s}{(1 + t_0^2) \mu_s} = a_0 = \left( \frac{1}{(1 + t_0^2) \mu_s} \right)^{1/2},
    \]
    which is a contradiction.

    \noindent\textbf{Case 2: $b_0 = 0$.}

    Then $a_0^2 = \mu_s^{-1}$, and
    \begin{align*}
        0 = \lim_{i \to \infty} \delta(u_i, v_i) &= 1 - g(\tilde{t}) \mu_s (\lambda a_0^p)^{2/p} \\
        &= 1 - \inf_{t \in (0, \infty)} g(t) \cdot \lambda^{2/p}.
    \end{align*}
    Hence, $\inf_{t \in (0, \infty)} g(t) = \lambda^{-2/p} = g(0)$. By \eqref{hypo2} with $\tilde{t} = 0$,
    \[
        \eta_i(0) = \| \nabla u_i \|_2^2 + \| \nabla v_i \|_2^2 - \sigma_i(0)^2 \mu_s \implies \sigma_i(0)^2 \le \frac{1}{\mu_s + \epsilon_1}.
    \]
    On the other hand, by Lemma~\ref{k sup} and \eqref{limfg},
    \begin{align*}
        \sigma_i(0) &\ge \frac{\int \nabla u_i \cdot \nabla f_i \,\mathrm{d}x}{\mu_s} = \frac{\int \nabla (u_i - a_i f_i) \cdot \nabla f_i \,\mathrm{d}x}{\mu_s} + a_i \to a_0 \quad \text{as } i \to \infty.
    \end{align*}
    Therefore,
    \[
        \frac{1}{(\mu_s + \epsilon_1)^{1/2}} \ge \limsup_{i\to \infty}\sigma_i(0) \ge a_0 = \frac{1}{\mu_s^{1/2}},
    \]
    which is again a contradiction.
\end{proof}

\section{ Sharp Quantitative Stability Estimates } \label{sec4}
Fix an $\epsilon_1>0$ to be chosen, and let $\epsilon_0$ be as in Theorem \ref{lem 222}. Noting that $(0,0)$ lies in the closure of $\mathcal{M}'$, which implies that
$$
    \inf_{(U_0,V_0)\in \mathcal{M}'} [\|\nabla(u-U_0)\|_2^2 + \|\nabla(v-V_0)\|_2^2]\leq \|\nabla u\|_2^2 + \|\nabla v\|_2^2,
$$
thus
$$
    \inf_{(U_0,V_0)\in \mathcal{M}'} \frac{\|\nabla(u-U_0)\|_2^2 + \|\nabla(v-V_0)\|_2^2}{\|\nabla u\|_2^2 + \|\nabla v\|_2^2}\leq 1.
$$
So, if $\delta(u,v)\ge \epsilon_0(\|\nabla u\|_2^2 + \|\nabla v\|_2^2)$, then the proof is done by taking $C=\epsilon_0^{-\iota}$.

Otherwise, $\delta(u,v)\leq \epsilon_0(\|\nabla u\|_2^2 + \|\nabla v\|_2^2)$, then by Theorem \ref{lem 222}, there exists $t_0\in\mathcal{S}$ s.t. the minimizer $(\sigma,w)$ of the problem \eqref{min pro} or \eqref{min infty} satisfies $\eta \le \epsilon_1 \sigma^2$, where $\eta=\eta(t_0)$ is defined as in Theorem \ref{lem 222}. Since $t_0=0$ and $t_0=\infty$ are symmetry, we may assume $t_0\in[0,\infty)$ in the following discussion.

Consider the operator $Lu=\frac{-|x|^s\Delta u}{w^{p-2}}$, then
\[
    Lw=\mu_sw,\quad L(\partial_\tau w)=(p-1)\mu_s\partial_\tau w,
\]
and by Lemma \ref{spectrum-HLS}, the first two eigenvalues of $L$ is $\Lambda_1=\mu_s$, $\Lambda_2=(p-1)\mu_s$, whose eigenspaces are spanned by $w$ and $\partial_\tau w$, respectively. We can decompose $u$ and $v$ as the following form:
\begin{equation}\label{decom}
    \begin{aligned}
        & u = aw + r\frac{\partial_\tau w}{k_1} + \sum_{i=3}^\infty r_i\varphi_i,\\
        & v = bw + q\frac{\partial_\tau w}{k_1} + \sum_{i=3}^\infty q_i\varphi_i,
    \end{aligned}
\end{equation}
here $k_1:= \left(\int \frac{1}{|x|^s}w^{p-2}(\partial_\tau w)^2\,\mathrm{d}x\right)^{1/2}$, $\varphi_i$ is the i-th eigenfunction of $L$ satisfying
\[
    L\varphi_i = \Lambda_i \varphi_i,\quad \int \frac{1}{|x|^s} w^{p-2}\varphi_i^2 \,\mathrm{d}x =1,
\]
and $\Lambda_i$ is the i-th eigenvalue of $L$, $\Lambda_i\ge \Lambda_3 > \Lambda_2=(p-1)\mu_s$. By \eqref{perp}, we additionally have
\begin{equation}
    \begin{cases}
        a+t_0b = (1+t_0^2)\sigma , \\
        r+t_0q = 0,
    \end{cases}
\end{equation}
which implies
\begin{equation}\label{abrs}
    t_0(b-t_0\sigma) = -(a-\sigma),\quad t_0q=-r.
\end{equation}
Moreover, since
\begin{align*}
    \|\nabla(u-\sigma w)\|_2^2 ={}& \|\nabla(u-aw + (a-\sigma)w)\|_2^2\\
    ={}& \|\nabla(u-aw)\|_2^2 + (a-\sigma)^2\mu_s \\
    \ge{}& (a-\sigma)^2\mu_s,\\
    \|\nabla(v-t_0\sigma w)\|_2^2 \ge{}& (b-t_0\sigma)^2\mu_s,
\end{align*}
by $\eta\le \epsilon_1 \sigma^2$ and \eqref{abrs}, we obtain that
\begin{equation}\label{ab_close}
    \begin{cases}
        (b/\sigma-t_0)^2 \le \frac{1}{(1+t_0^2)\mu_s}\epsilon_1,\\
        (a/\sigma-1)^2 \le \frac{t_0^2}{(1+t_0^2)\mu_s}\epsilon_1.
    \end{cases}
\end{equation}
This shows that $(a/\sigma,b/\sigma)$ is close to $(1,t_0)$. We mention that $a,b\ge 0$ as well:
\[
    a=\int \frac{1}{|x|^s}w^{p-1}u\,\mathrm{d}x \ge 0,\quad b=\int \frac{1}{|x|^s}w^{p-1}v\,\mathrm{d}x \ge 0.
\]
Denote 
\begin{equation}\label{eq:defi-gamma}
    \gamma := \|\nabla (u-aw)\|_2^2 + \|\nabla(v-bw)\|_2^2,
\end{equation}
then $\gamma \le \eta \le \epsilon_1 \sigma^2$. Use the elementary inequality $x^{2/p}-1 \le \frac 2p(x-1)$ (for $x>0$), we obtain
\begin{align}
    \gamma ={}& \|\nabla u\|^2 + \|\nabla v\|^2 - (a^2+b^2)\mu_s \notag \\
    ={}& \delta(u,v) + g(t_0)\mu_s \left(\int \frac{1}{|x|^s} (\lambda u^p + \mu v^p + \kappa p u^\alpha v^\beta)\,\mathrm{d}x\right)^{2/p} \notag \\
    &- g(b/a)\mu_s (\lambda a^p + \mu b^p + \kappa p a^\alpha b^\beta)^{2/p} \notag \\
    ={}& \delta(u,v)- (g(b/a)-g(t_0))B^{2/p} {\mu_s} \notag\\
    & + g(t_0)\mu_s \left(\left(\int \frac{1}{|x|^s} (\lambda u^p + \mu v^p + \kappa p u^\alpha v^\beta)\,\mathrm{d}x\right)^{2/p} - B^{2/p}\right) \notag \\
    \le{}& \delta(u,v)- (g(b/a)-g(t_0))B^{2/p} {\mu_s}  \notag\\
    & + g(t_0)\mu_s \frac 2p B^{2/p-1} \left( \int \frac{1}{|x|^s} (\lambda u^p + \mu v^p + \kappa p u^\alpha v^\beta)\,\mathrm{d}x - B\right), \label{est*}
\end{align}
where
\[
    B := \int \frac{1}{|x|^s} \big( \lambda (aw)^p + \mu (bw)^p + \kappa p (aw)^\alpha(bw)^\beta \big)\,\mathrm{d}x = \lambda a^p +\mu b^p + \kappa p a^\alpha b^\beta.
\]
By $a+t_0b = (1+t_0^2)\sigma$ and $a,b\ge 0$, we directly obtain $\max\{a,b\}\approx \sigma$ and so $B\approx \sigma^p$.

\begin{lemma}\label{Lemma_t>0}
    Let $a,b$ be given by \eqref{decom} and $\gamma$ be defined by \eqref{eq:defi-gamma}. If $t_0>0$, then for sufficiently small $\epsilon_1$, we have
    \begin{equation}\label{gamma}
        \frac{\gamma}{\sigma^2} + (g(b/a)-g(t_0)) \lesssim \frac{\delta(u,v)}{\sigma^2}.
    \end{equation}
\end{lemma}
\begin{proof}
    If $t_0>0$, by \eqref{ab_close}, for sufficiently small $\epsilon_1$,
    \begin{equation}\label{ab_close_1}
        \begin{cases}
            |b/\sigma-t_0| \le \sqrt{\frac{\epsilon_1}{(1+t_0^2)\mu_s}} < \frac 12 t_0,\\
            |a/\sigma-1| \le \sqrt{ \frac{t_0^2\epsilon_1}{(1+t_0^2)\mu_s}} < \frac 12.
        \end{cases}
    \end{equation}
Using Lemma \ref{lem 1}, fix a $m>0$ to be determined later, we have
    \begin{equation}\label{est 0.1}
        \begin{aligned}
            & \int \frac{1}{|x|^s}(u^p - (aw)^p)\,\mathrm{d}x\\
            &\le p \int \frac{1}{|x|^s} (u-aw)(aw)^{p-1}\,\mathrm{d}x \\
            &\quad + \left(\frac{p(p-1)}{2}+m\right)\int \frac{1}{|x|^s}\Big((aw)^{p-2}(u-aw)^2\,\mathrm{d}x + C_1|u-aw|_p^p\Big)\,\mathrm{d}x \\
            &\le \left(\frac{p(p-1)}{2}+m\right)a^{p-2}\int \frac{1}{|x|^s}w^{p-2}(u-aw)^2\,\mathrm{d}x + C_1\mu_s^{-p/2}\|\nabla(u-aw)\|_2^p.
        \end{aligned}
    \end{equation}
    Here by the decomposition \eqref{decom} of $u$, $\int \frac{1}{|x|^s} (u-aw)w^{p-1}\,\mathrm{d}x = 0$. Similarly,
    \begin{equation}\label{est 0.2}
        \begin{aligned}
            & \int \frac{1}{|x|^s}(v^p - (bw)^p)\,\mathrm{d}x\\
            &\le \left(\frac{p(p-1)}{2}+m\right) b^{p-2}\int \frac{1}{|x|^s}w^{p-2}(v-bw)^2\,\mathrm{d}x + C_1  {\mu_s^{-p/2}}\|\nabla(v-bw)\|_2^p.
        \end{aligned}
    \end{equation}
    To deal with the term $u^\alpha v^\beta$, we use Lemma \ref{Lem2} to derive
    \begin{equation}\label{inter term_0}
        \begin{aligned}
            & \int \frac{1}{|x|^s} (u^\alpha v^\beta - (aw)^\alpha(bw)^\beta)\,\mathrm{d}x\\
            &\le \alpha a^{\alpha-1}b^{\beta} \int \frac{1}{|x|^s} w^{p-1}(u-aw)\,\mathrm{d}x + \beta a^{\alpha}b^{\beta-1} \int\frac{1}{|x|^s}w^{p-1}(v-bw)\,\mathrm{d}x\\
            & \quad + \left(\frac{\alpha(\alpha-1)}{2}+m\right)a^{\alpha-2}b^\beta\int \frac{1}{|x|^s}w^{p-2}(u-aw)^2\,\mathrm{d}x \\
            & \quad + \left(\frac{\beta(\beta-1)}{2}+m\right)a^{\alpha}b^{\beta-2}\int \frac{1}{|x|^s} w^{p-2}(v-bw)^2\,\mathrm{d}x\\
            & \quad + \alpha\beta a^{\alpha-1}b^{\beta-1}\int \frac 1{|x|^s}w^{p-2}(u-aw)(v-bw)\,\mathrm{d}x + \kappa p\, C_3 f(\alpha,\beta, u,v).
        \end{aligned}
    \end{equation}
    Here  $f(\alpha,\beta,u,v)$ is the resident term shown in Lemma \ref{lem 1}. In detail, if $\alpha,\beta>2$,
    \begin{align*}
        & f(\alpha,\beta,u,v) \\
        &= \int \frac{1}{|x|^s} \Big((aw)^\alpha |v-bw|^\beta + (bw)^\beta |u-aw|^\alpha+ (aw)^{\alpha-1}(bw)^{\beta-2} |u-aw||v-bw|^2 \\ & \qquad + (aw)^{\alpha-2}(bw)^{\beta-1}|u-aw|^2|v-bw| + |u-aw|^\alpha|v-bw|^\beta \Big)\,\mathrm{d}x \\
        &\le a^\alpha \mu_s^{-\beta/2}\|\nabla(v-bw)\|_2^\beta \\
        & \quad + b^\beta \mu_s^{-\alpha/2} \|\nabla(u-aw)\|_2^\alpha + a^{\alpha-1}b^{\beta-2} \mu_s^{-3/2}\|\nabla(u-aw)\|_2\|\nabla(v-bw)\|^2_2 \\
        & \quad + a^{\alpha-2}b^{\beta-1}\|\nabla(u-aw)\|_2^2\|\nabla(v-bw)\|_2 + \mu_s^{-p/2}\|\nabla(u-aw)\|_2^\alpha\|\nabla(v-bw)\|^\beta_2 \\
        &\lesssim (1+O(1)\epsilon_1)^\alpha \sigma^\alpha\gamma^{\beta/2} + (t_0+O(1)\epsilon)^\beta \sigma^\beta \gamma^{\alpha/2}\\
        & \quad + (1+O(1)\epsilon_1)^{\alpha-1}(t_0+O(1)\epsilon_1)^{\beta-2} \sigma^{p-3}\gamma^{3/2} \\
        & \quad + (1+O(1)\epsilon_1)^{\alpha-2}(t_0+O(1)\epsilon_1)^{\beta-1} {\sigma^{p-3}}\gamma^{3/2} + \gamma^{p/2}\\
        &\lesssim \sigma^\alpha \gamma^{\beta/2} + \sigma^\beta \gamma^{\alpha/2} + \sigma^{p-3}\gamma^{3/2} + \gamma^{p/2}.
    \end{align*}
    Similarly for other cases, we can get

    \noindent $f(\alpha,\beta,u,v)\lesssim $
    \[ \begin{cases}
        \sigma^\alpha \gamma^{\beta/2} + \sigma^\beta \gamma^{\alpha/2} + \sigma^{p-3}\gamma^{3/2} + \gamma^{p/2}, & \text{if}\,\min\{\alpha,\beta\}>2;\\
        \sigma^2\gamma^{\max\{\alpha,\beta\}/2} + \sigma^{\max\{\alpha,\beta\}-1}\gamma^{3/2} + \gamma^{p/2}, & \text{if}\,\min\{\alpha,\beta\}=2<\max\{\alpha,\beta\};\\
        \sigma\gamma^{3/2} + \gamma^{2}, & \text{if}\,\alpha=\beta=2;\\
        \sigma^{\min\{\alpha,\beta\}} \gamma^{\max\{\alpha,\beta\}/2} + \sigma^{p-3}\gamma^{3/2} \\
        \quad + \sigma^{\max\{\alpha,\beta\}-2} \gamma^{1+\min\{\alpha,\beta\}/2} \\
        \quad + \sigma^{\max\{\alpha,\beta\}-1}\gamma^{\frac{\min\{\alpha,\beta\}+1}{2}} + \gamma^{p/2}, & \text{if}\,\min\{\alpha,\beta\}< 2 <\max\{\alpha,\beta\}; \\
        \sigma^{p-3}\gamma^{3/2} + \sigma\gamma^{\frac{\min\{\alpha,\beta\}+1}{2}} + \gamma^{p/2}, & \text{if}\,\min\{\alpha,\beta\} < 2 = \max\{\alpha,\beta\}; \\
        \sigma^{\alpha-1}\gamma^{\frac{\beta+1}{2}} + \sigma^{\beta-1}\gamma^{\frac{\alpha+1}{2}} + \sigma^{\alpha+\frac{\beta}{2} - \frac 32} \gamma^{\frac{\beta+3}{4}}\\
        \quad + \sigma^{\beta+\frac{\alpha}{2} - \frac 32} \gamma^{\frac{\alpha+3}{4}} + \gamma^{p/2}, & \text{if}\,\max\{\alpha,\beta\}<2.
    \end{cases} \]

    By \eqref{est 0.1}-\eqref{inter term_0},
    \begin{align*}
        & \int \frac{1}{|x|^s}(\lambda u^p + \mu v^p + \kappa p u^\alpha  v^\beta )\,\mathrm{d}x - B \\
        &\le k^{p-2} \bigg( A_u\int \frac{1}{|x|^s}w^{p-2}(u-aw)^2\,\mathrm{d}x +  A_v\int \frac{1}{|x|^s}w^{p-2}(v-bw)^2\,\mathrm{d}x\\
        & \qquad + A_{u,v} \int \frac 1{|x|^s}w^{p-2}(u-aw)(v-bw)\,\mathrm{d}x\bigg) \\
        &\quad + C_1\mu_s^{-p/2}\gamma^{p/2} +\kappa p \,C_3 f(\alpha,\beta, u,v).
    \end{align*}
    Where
    \begin{align}
        A_u = {}& \lambda (a/\sigma)^{p-2} \left(\frac{p(p-1)}{2}+m\right) + \kappa p (a/\sigma)^{\alpha-2}(b/\sigma)^{\beta}\left(\frac{\alpha(\alpha-1)}{2}+ m\right), \label{0.A} \\
        A_v = {}& \mu (b/\sigma)^{p-2} \left(\frac{p(p-1)}{2}+m\right)  + \kappa p (a/\sigma)^{\alpha}(b/\sigma)^{\beta-2} \left(\frac{\beta(\beta-1)}{2}+ m\right), \label{0.B} \\
        A_{u,v} = {}& \kappa p \alpha\beta (a/\sigma)^{\alpha-1}(b/\sigma)^{\beta-1}. \label{0.C}
    \end{align}
    Taking all the estimates into \eqref{est*} yields
    \begin{equation}\label{est0.111}
        \begin{aligned}
            & \gamma + (g(b/a)-g(t_0))B^{2/p}  {\mu_s}\\
            &\le \delta(u,v) + g(t_0) \mu_s \frac{2}{p} B^{2/p-1} \bigg(\int \frac{1}{|x|^s}(\lambda u^p + \mu v^p + \kappa p u^\alpha  v^\beta )\,\mathrm{d}x - B \bigg)\\
            &\le \delta(u,v) + g(t_0) \mu_s \frac 2p (B/  {\sigma^p})^{2/p-1} \bigg(A_u\int \frac{1}{|x|^s}w^{p-2}(u-aw)^2\,\mathrm{d}x \\ & \qquad +  A_v\int \frac{1}{|x|^s}w^{p-2}(v-bw)^2\,\mathrm{d}x + A_{u,v} \int \frac 1{|x|^s}w^{p-2}(u-aw)(v-bw)\,\mathrm{d}x\bigg) \\
            & \quad + C \sigma^{2-p}(\gamma^{p/2} + f(\alpha,\beta,u,v)).
        \end{aligned}
    \end{equation}
    Denote
    \begin{align}
        A_u^0 = {}& \lambda \left(\frac{p(p-1)}{2}+m\right) + \kappa p \left(\frac{\alpha(\alpha-1)}{2}+ m\right) t_0^\beta, \label{A} \\
        A_v^0 = {}& \mu t_0^{p-2} \left(\frac{p(p-1)}{2}+m\right)  + \kappa p t_0^{\beta -2} \left(\frac{\beta(\beta-1)}{2}+ m\right), \label{B} \\
        A_{u,v}^0 = {}& \kappa p \alpha \beta t_0^{\beta-1}. \label{C}
    \end{align}
    Since $\frac{x^\iota-1}{x-1}$ is increasing if $\iota\ge 1$ or $\iota\le 0$, and decreasing if $0<\iota<1$, for every $x\in (1/2,3/2)$,
    \begin{equation}\label{0.x^p-1}
        \left|\frac{x^\iota-1}{x-1}\right| \le
        \begin{cases}
            2((3/2)^\iota -1), & \iota\ge 1,\\
            1,& 0<\iota<1,\\
            2^{1-p}-2, & \iota<0.
        \end{cases}
    \end{equation}
    As a result, by \eqref{ab_close_1}, compared \eqref{0.A}-\eqref{0.C} to \eqref{A}-\eqref{C},
    \begin{equation}\label{0.(A')close(A)}
        |A_u-A_u^0| + |A_v-A_v^0| + |A_{u,v}-A_{u,v}^0| \lesssim |a/\sigma-1| + |b/\sigma-t_0| \lesssim \epsilon_1^{1/2}.
    \end{equation}
    Similarly,
    \begin{equation}\label{0.coefficient}
        \begin{aligned}
            & \left|g(t_0)(B/  {\sigma^p})^{2/p-1} - \frac{1+t_0^2}{\lambda + \mu t_0^p + \kappa p t_0^\beta}\right| \\
            &= \frac{1+t_0^2}{\lambda + \mu t_0^p + \kappa p t_0^\beta }\left|\left(\frac{\lambda(a/\sigma)^p + \mu(b/\sigma)^p + \kappa p (a/\sigma)^\alpha(b/\sigma)^\beta}{\lambda+\mu t_0^p +\kappa p t_0^\beta}\right)^{2/p-1}-1\right| \\
            &\lesssim |a/\sigma-1| + |b/\sigma-t_0| \\
            &\lesssim \epsilon_1^{1/2}.
        \end{aligned}
    \end{equation}
    We claim, as will be shown below, that for sufficiently small $m$, there exists a constant $c_1 = c_1(\lambda, \mu, \kappa, \alpha, \beta, N, s, m)\in (0,1)$ such that
    \begin{equation}\label{0 est of 2nd term}
        \begin{aligned}
            c_1\gamma \ge{}& \mu_s \frac 2p \cdot \frac{1+t_0^2}{\lambda + t_0^p + \kappa p t_0^\beta} \bigg(A_u^0\int \frac{1}{|x|^s}w^{p-2}(u-aw)^2\,\mathrm{d}x \\ & \qquad +  A_v^0\int \frac{1}{|x|^s}w^{p-2}(v-bw)^2\,\mathrm{d}x + A_{u,v}^0 \int \frac 1{|x|^s}w^{p-2}(u-aw)(v-bw)\,\mathrm{d}x\bigg).
        \end{aligned}
    \end{equation}
    Taking \eqref{0.(A')close(A)}-\eqref{0 est of 2nd term} into \eqref{est0.111},
    \begin{align*}
        & \gamma + (g(b/a)-g(t_0))B^{2/p} {\mu_s} \\
        &\le \delta(u,v) + \mu_s \frac 2p \cdot \frac{1+t_0^2}{\lambda + t_0^p + \kappa p t_0^\beta} \bigg(A_u^0\int \frac{1}{|x|^s}w^{p-2}(u-aw)^2\,\mathrm{d}x \\
        & \qquad  +A_v^0\int \frac{1}{|x|^s}w^{p-2}(v-bw)^2\,\mathrm{d}x + A_{u,v}^0 \int \frac 1{|x|^s}w^{p-2}(u-aw)(v-bw)\,\mathrm{d}x\bigg) \\
        & \quad + C \epsilon_1^{1/2}\left(\int \frac{1}{|x|^s}\big(w^{p-2}(u-aw)^2 + w^{p-2}(v-bw)^2\big)\,\mathrm{d}x\right) \\
        & \quad + C \sigma^{2-p}(\gamma^{p/2} + f(\alpha,\beta,u,v)) \\
        &\le \delta(u,v) + c_1\gamma + C\epsilon^{1/2}\gamma + C\sigma^{2-p}(\gamma^{p/2}+f(\alpha,\beta,u,v)).
    \end{align*}
    Thus, if $\epsilon_1$ is sufficiently small such that $c_1+C\epsilon_1^{1/2}<1$, we can deduce that
    \[
        \gamma + (g(b/a)-g(t_0))\sigma^{2} \lesssim \delta(u,v) + \sigma^{2-p}(\gamma^{p/2}+f(\alpha,\beta,u,v)).
    \]
    and hence,
    \begin{align*}
        & \frac{\gamma}{\sigma^2} + (g(b/a)-g(t_0))  \lesssim{} \\
        & \begin{cases}
            \frac{\delta(u,v)}{\sigma^2} + \left(\frac{\gamma}{\sigma^2}\right)^{p/2} + \left(\frac{\gamma}{\sigma^2}\right)^{\beta/2} + \left(\frac{\gamma}{\sigma^2}\right)^{\alpha/2} + \left(\frac{\gamma}{\sigma^2}\right)^{3/2},   & \text{if}\, \min\{\alpha,\beta\}>2; \\[5pt]
            \frac{\delta(u,v)}{\sigma^2} + \left(\frac{\gamma}{\sigma^2}\right)^{\max\{\alpha,\beta\}/2} + \left(\frac{\gamma}{\sigma^2}\right)^{3/2} + \left(\frac{\gamma}{\sigma^2}\right)^{p/2}, & \text{if}\, \min\{\alpha,\beta\} = 2 < \max\{\alpha,\beta\};\\[5pt]
            \frac{\delta(u,v)}{\sigma^2} + \left(\frac{\gamma}{\sigma^2}\right)^{3/2} + \left(\frac{\gamma}{\sigma^2}\right)^2,& \text{if}\, \alpha=\beta=2;\\[5pt]
            \frac{\delta(u,v)}{\sigma^2} + \left(\frac{\gamma}{\sigma^2}\right)^{p/2} + \left(\frac{\gamma}{\sigma^2}\right)^{\max\{\alpha,\beta\}/2} + \left(\frac{\gamma}{\sigma^2}\right)^{3/2}\\
            \quad + \left(\frac{\gamma}{\sigma^2}\right)^{\min\{\alpha,\beta\}/2+1} + \left(\frac{\gamma}{\sigma^2}\right)^{\frac{\min\{\alpha,\beta\}+1}{2}},&\text{if}\, \min\{\alpha,\beta\}<2<\max\{\alpha,\beta\}; \\[5pt]
            \frac{\delta(u,v)}{\sigma^2} + \left(\frac{\gamma}{\sigma^2}\right)^{p/2} + \left(\frac{\gamma}{\sigma^2}\right)^{3/2} + \left(\frac{\gamma}{\sigma^2}\right)^{\frac{\min\{\alpha,\beta\}+1}{2}},&   \text{if}\, \min\{\alpha,\beta\}<2=\max\{\alpha,\beta\}; \\[5pt]
            \frac{\delta(u,v)}{\sigma^2} + \left(\frac{\gamma}{\sigma^2}\right)^{p/2} + \left(\frac{\gamma}{\sigma^2}\right)^{\frac{\beta+1}2} + \left(\frac{\gamma}{\sigma^2}\right)^{\frac{\alpha+1}{2}} \\ \quad +\left(\frac{\gamma}{\sigma^2}\right)^{\frac{\beta+3}{4}} + \left(\frac{\gamma}{\sigma^2}\right)^{\frac{\alpha+3}{4}},&\text{if}\, \max\{\alpha,\beta\}<2.
        \end{cases}
    \end{align*}
    In each case, all exponents of $\frac{\gamma}{\sigma^2}$ on the right side are greater than $1$, showing that if $\frac{\gamma}{\sigma^2}\le \epsilon_1$ is sufficiently small, all these terms can be absorbed by the left side term. Thus, to conclude the proof, we only need to prove \eqref{0 est of 2nd term}.
\end{proof}

\begin{proof}[\bf Proof of \eqref{0 est of 2nd term}]
    From \eqref{decom} and \eqref{abrs},
    \begin{align*}
        \int \frac{1}{|x|^s}w^{p-2}(u-aw)^2\,\mathrm{d}x ={}& r^2 + \sum_{i=3}^\infty r_i^2, \\
        \int \frac{1}{|x|^s}w^{p-2}(v-bw)^2\,\mathrm{d}x ={}& t_0^{-2}r^2 + \sum_{i=3}^\infty q_i^2, \\
        \int \frac 1{|x|^s}w^{p-2}(u-aw)(v-bw)\,\mathrm{d}x ={}& r  {q} + \sum_{k=i}^\infty r_iq_i, \\
        ={}& - t_0^{-1}r^2 + \sum_{i=3}^\infty r_iq_i\\
        \|\nabla(u-aw)\|_2^2 ={}& (p-1)\mu_s r^2 + \sum_{i=3}^\infty \Lambda_i r_i^2 \\
        \|\nabla(v-bw)\|_2^2 ={}& (p-1)\mu_s t_0^{-2}r^2 + \sum_{i=3}^\infty \Lambda_i q_i^2.
    \end{align*}
    Thus, \eqref{0 est of 2nd term} equals to
    \begin{equation}
        \begin{aligned}
            & c_1 (1+t_0^{-2})(p-1)r^2 + c_1\sum_{i=3}^\infty \frac{\Lambda_i}{\mu_s}(r_i^2+q_i^2)\\
            & \ge  \frac 2p \cdot \frac{1+t_0^2}{\lambda + \mu t_0^p + \kappa p t_0^\beta} \cdot \\
            & \bigg( (A_u^0+t_0^{-2}A_v^0 - t_0^{-1}A_{u,v}^0)r^2 + \sum_{i=3}^\infty \big(A_u^0r_i^2+A_v^0q_i^2 + A_{u,v}^0r_iq_i\big)\bigg).
        \end{aligned}
    \end{equation}
    This is true if we let
    \begin{align}
        c_1(1+t_0^{-2})(p-1) \ge{}& \frac 2p \cdot \frac{1+t_0^2}{\lambda + \mu t_0^p + \kappa p t_0^\beta}(A_u^0+t_0^{-2}A_v^0 - t_0^{-1}A_{u,v}^0),\label{hope0.1}\\
        c_1(\Lambda_i\mu_s^{-1}(r_i^2+q_i^2)) \ge{}&  \frac 2p \cdot \frac{1+t_0^2}{\lambda + \mu t_0^p + \kappa p t_0^\beta}(A_u^0r_i^2 + A_v^0q_i^2 + A_{u,v}^0r_iq_i),\quad i\ge 3.\label{hope0.2}
    \end{align}
    Then, we have to rigorously compute $A_u^0,A_v^0,A_{u,v}^0$. By \eqref{A},\eqref{B} and \eqref{g'=0},
    \begin{align*}
        A_u^0 ={}& \frac{p(p-1)}{2} \left(\lambda + \kappa\alpha t_0^\beta \right) - \frac 12 \kappa p \alpha\beta t_0^\beta + (\lambda + \kappa p t_0^\beta) m, \\
        ={}& \frac{p(p-1)}{2} \left(\lambda + \kappa\alpha t_0^\beta \right) - \frac 12 \kappa p \alpha\beta t_0^\beta + O(1)m;\\
        A_v^0 ={}& \frac{p(p-1)}{2}\left(\mu t_0^{p-2} + \kappa\beta t_0^{\beta-2}\right) - \frac 12 \kappa p \alpha\beta t_0^{\beta-2} + (\mu t_0^{p-2} + \kappa p t_0^{\beta-2}) m \\
        ={}& \frac{p(p-1)}{2} \left(\lambda + \kappa\alpha t_0^\beta \right) - \frac 12 \kappa p \alpha\beta t_0^{\beta-2} + O(1)m.
    \end{align*}
    Hence, the right side of \eqref{hope0.1} becomes
    \begin{align*}
        & \frac 2p \cdot \frac{1+t_0^2}{\lambda + \mu t_0^p + \kappa p t_0^\beta}(A_u^0+t_0^{-2}A_v^0 - t_0^{-1}A_{u,v}^0) \\
        &= \frac 2p \cdot \frac{1}{\lambda + \kappa\alpha t_0^{\beta}}\cdot \\
        & \quad \left(\frac{p(p-1)}{2} (1+t_0^{-2})\left(\lambda + \kappa\alpha t_0^\beta \right) - \frac 12 \kappa p \alpha\beta t_0^\beta(1+t_0^{-4}+2t_0^{-2}) + O(1)m\right) \\
        &= (1+t_0^{-2})(p-1) - (1+t_0^{-2})^2\frac{\kappa\alpha\beta t_0^{\beta}}{\lambda + \kappa \alpha t_0^\beta} + O(1)m,
    \end{align*}
    and \eqref{hope0.1} equals to
    \begin{equation}\label{0.c11}
        (p-1)c_1 \ge (p-1) - (1+t_0^{-2})\frac{\kappa\alpha\beta t_0^{\beta}}{\lambda + \kappa \alpha t_0^\beta} + O(1)m.
    \end{equation}
    By \eqref{g''>=0},
    \begin{align*}
        (p-1) - (1+t_0^{-2})\frac{\kappa\alpha\beta t_0^{\beta}}{\lambda + \kappa \alpha t_0^\beta} ={}& \frac{(p-2)\lambda + \kappa\alpha t_0^\beta (p-2) - \kappa\alpha\beta t_0^\beta - \kappa\alpha \beta t_0^{\beta-2}}{\lambda + \kappa\alpha t_0^\beta} +1\\
        ={}& \frac{(p-2)\lambda - \alpha(2-\alpha)\kappa t_0^\beta - \kappa t_0^{\beta-2}\alpha\beta }{\lambda + \kappa\alpha t_0^\beta}+1\\ \le{}& 1,
    \end{align*}
    and so, since $p-1>1$, for sufficiently small $m$, the right side of \eqref{0.c11} is less than $p-1$. Similarly, the right side of \eqref{hope0.2} becomes
    \begin{align*}
        & \frac 2p \cdot \frac{1+t_0^2}{\lambda + \mu t_0^p + \kappa p t_0^\beta}(A_u^0r_i^2 + A_v^0q_i^2 + A_{u,v}^0r_iq_i) \\
        &=  \frac 2p \cdot \frac{1}{\lambda + \kappa\alpha t_0^{\beta}}\cdot\\
        &\quad \left( \frac{p(p-1)}{2}\left(\lambda + \kappa\alpha t_0^\beta \right) (r_i^2 + q_i^2) - \frac 12 \kappa p \alpha\beta t_0^{\beta-2}( t_0^2r_i^2 + q_i^2 - 2t_0r_iq_i)\right)\\
        &\quad + O(1)m(r_i^2+q_i^2) \\
        &\le (p-1+O(1)m)(r_i^2+q_i^2),
    \end{align*}
    and combining $\Lambda_i\ge \Lambda_3$, $\forall\, i \ge 3$, \eqref{hope0.2} holds for every $i\ge 3$ if
    \begin{equation}\label{0.c12}
        c_1 \ge \frac{p-1 + O(1)m}{\Lambda_3\mu_s^{-1}}.
    \end{equation}
    Since $\Lambda_3>(p-1)\mu_s$, for sufficiently small $m$, the right side of \eqref{0.c12} is less than $1$, and finally, we can pick some $0<c_1<1$ that satisfies \eqref{0 est of 2nd term}.
\end{proof}

\begin{lemma}\label{Lemma_t=0}
    Let $a,b$ be given by \eqref{decom} and $\gamma$ be defined by \eqref{eq:defi-gamma}. If $t_0=0$, then $a=\sigma$, and for sufficiently small $\epsilon_1$, we still have
    \begin{equation}\label{est2-4}
        \frac{\gamma}{\sigma^2} + (g(b/a)-g(0)) \lesssim \frac{\delta(u,v)}{\sigma^2}.
    \end{equation}
\end{lemma}
\begin{proof}
    In this case, $0$ is a minimizer of $g(t)$, implying that $\inf_{t\in(0,\infty)}g(t) = \lambda^{-2/p}$. Hence, Proposition \ref{prop g(0)} tells that $\beta \ge 2$. By \eqref{abrs}, $a=\sigma,\, r=0$ and so by the spectrum analysis,
    \begin{equation}\label{est2-1}
        \begin{aligned}
            \int \frac{1}{|x|^s}w^{p-2}(u-\sigma w)^2 \,\mathrm{d}x \le{}& \frac{1}{\Lambda_3} \|\nabla (u-\sigma w)\|_2^2. \\
            \int \frac{1}{|x|^s}w^{p-2}(v-bw)^2 \,\mathrm{d}x \le{}& \frac{1}{(p-1)\mu_s} \|\nabla (v-bw)\|_2^2.
        \end{aligned}
    \end{equation}
    Moreover, \eqref{ab_close} shows that
    \begin{equation}\label{est2-2}
        (b/\sigma)^2 \le \mu_s^{-1}\epsilon_1.
    \end{equation}
    By \eqref{est*},
    \begin{align*}
        \gamma :=& \|\nabla (u-a w)\|_2^2 + \|\nabla(v-bw)\|_2^2 \\
        \le{}& \delta(u,v) - (g(b/\sigma)-g(0))\sigma^2(\lambda + \mu (b/\sigma)^p + \kappa p (b/\sigma)^\beta)^{2/p} \\
        & + \Big[\mu_s \frac 2p \sigma^{2-p}\Big(1 + \frac{\mu}{\lambda} (b/\sigma)^p + \frac{\kappa p}{\lambda} (b/\sigma)^\beta\Big)^{2/p-1}\lambda^{-1}\Big] \\
        & \quad\quad \cdot \Big[ \int \frac{1}{|x|^s} \Big(\lambda (u^p -(\sigma w)^p) + \mu (v^p-(bw)^p) \\
        & \qquad\qquad + \kappa p (u^\alpha v^\beta - (\sigma w)^\alpha(bw)^\beta)\Big) \,\mathrm{d}x \Big].
    \end{align*}
  Noting that $a=\sigma$, which implies that $\int \frac{1}{|x|^s}(u-\sigma w)w^{p-1}\,\mathrm{d}x=0$. So, by Lemma \ref{lem 1} and \eqref{est2-1}-\eqref{est2-2},
    \begin{align*}
        &\int \frac{1}{|x|^s} (u^p -(\sigma w)^p)\,\mathrm{d}x \\
        & \le  \sigma^{p-2} \left(\frac{p(p-1)}{2}+m\right)\int\frac{1}{|x|^s}w^{p-2}(u-\sigma w)^2 + C \int\frac{1}{|x|^s}|u-\sigma w|^p\,\mathrm{d}x \\
        & \le \sigma^{p-2}\left(\frac{p(p-1)}{2\Lambda_3}+\frac{m}{\Lambda_3}\right)\|\nabla (u-\sigma w)\|_2^2 + C \gamma^{p/2}, \\\\
        & \int \frac{1}{|x|^s} (v^p -(bw)^p)\,\mathrm{d}x \\
        & \le  b^{p-2}\left(\frac{p(p-1)}{2}+m\right)\int\frac{1}{|x|^s}w^{p-2}(v-bw)^2 + C \|\nabla(v-bw)\|_2^p \\
        & \le \sigma^{p-2}(\mu_s^{-1}\epsilon_1)^{p/2-1}\left(\frac{p(p-1)}{2}+m\right)\frac{1}{(p-1)\mu_s}\|\nabla(v-bw)\|_2^2 + C\|\nabla(v-bw)\|_2^p \\
        & \le C(\epsilon_1^{p/2-1}\gamma + \gamma^{p/2}).
    \end{align*}
    By Lemma \ref{Lem2} and \eqref{est2-1}-\eqref{est2-2},
    \begin{align*}
        & \int \frac{1}{|x|^s} (u^\alpha v^\beta - (\sigma w)^\alpha(bw)^\beta)\,\mathrm{d}x \\
        &\le \left(\frac{\alpha(\alpha-1)}{2}+m\right) \sigma^{\alpha-2}b^\beta\int \frac{1}{|x|^s}w^{p-2}(u-\sigma w)^2\,\mathrm{d}x\\
        &\quad + \left(\frac{\beta(\beta-1)}{2}+m\right)\sigma^\alpha b^{\beta-2}\int \frac{1}{|x|^s}w^{p-2}(v-bw)^2\,\mathrm{d}x \\
        &\quad + \alpha\beta \sigma^{\alpha-1}b^{\beta-1} \int \frac{1}{|x|^s} w^{p-2}(u-\sigma w)(v-bw)\,\mathrm{d}x + C_3 f(\alpha,\beta,u,v) \\
        &\le \frac{1}{(p-1)\mu_s}\left(\frac{\beta(\beta-1)}{2}+m\right)\sigma^{p-2} (\mu_s^{-1}\epsilon_1)^{\beta/2-1}\|\nabla(v-bw)\|_2^2 + C (\epsilon_1^{\beta/2}+\epsilon_1^{\frac{\beta-1}{2}})\gamma \\
        &\quad + C_3 f(\alpha,\beta,u,v),
    \end{align*}
    where
    \begin{multline*}
        f(\alpha,\beta,u,v) \lesssim{} \\
        \begin{cases}
            \sigma^\alpha \gamma^{\beta/2} + b^\beta \gamma^{\alpha/2} + (\sigma^{\alpha-1}b^{\beta-2}+\sigma^{\alpha-2}b^{\beta-1}) \gamma^{3/2} + \gamma^{p/2}, & \text{if}\, \min\{\alpha,\beta\} >2;\\
            \sigma^2\gamma^{\beta/2} + (kb^{\beta-2}+b^{\beta-1})\gamma^{3/2} + \gamma^{p/2}, & \text{if}\,\alpha=2<\beta; \\
            b^2\gamma^{\alpha/2} + (\sigma^{\alpha-2}b+\sigma^{\alpha-1})\gamma^{3/2} + \gamma^{p/2}, & \text{if}\,\alpha>2 = \beta; \\
            (k+b)\gamma^{3/2} + \gamma^{p/2}, & \text{if}\,\alpha=\beta=2;\\
            \sigma^\alpha \gamma^{\beta/2} + (\sigma^{\alpha-1}b^{\beta-2}+\sigma^{\alpha-2}b^{\beta-1})\gamma^{3/2}\\
             \phantom{\sigma^\alpha \gamma^{\beta/2}} +\sigma^{-1}b^{\beta-1}\gamma^{\frac{\alpha+3}{4}} + \gamma^{p/2}, &\text{if}\, \alpha<2<\beta;\\
            \sigma^{-1}b^2\gamma^{\frac{\alpha+1}2} + \sigma^{\alpha-1}\gamma^{3/2} + (\sigma^{-1}b + 1)\gamma^{p/2}, & \text{if}\,\alpha<2=\beta.
        \end{cases}
    \end{multline*}
    Hence,
    \begin{align}
        & \gamma + (g(b/\sigma)-g(0)) \sigma^2\lambda^{2/p} \notag\\
        &\le \delta(u,v) + \left(1+\frac{\mu}{\lambda}(b/\sigma)^{p} + \frac{\kappa p}{\lambda}(b/\sigma)^{\beta}\right)^{2/p-1} \notag \\
        & \quad \cdot \bigg( \left(\frac{p-1}{\Lambda_3/\mu_s} + \frac{2\mu_s m}{p\Lambda_3}\right) \|\nabla(u-\sigma w)\|_2^2 + \notag \\
        & \qquad \quad (\mu_s^{-1}\epsilon_1)^{\beta/2-1}\frac{\sigma}{(p-1)\lambda}\left(\beta(\beta-1) + 2m\right)\|\nabla(v-bw)\|_2^2\bigg) \notag \\
        & + C_4 \Big((\epsilon_1^{\beta/2}+\epsilon_1^{\frac{\beta-1}{2}} + \epsilon_1^{p/2-1})\gamma + \gamma^{p/2} + f(\alpha,\beta,u,v)\Big). \label{est2-3}
    \end{align}
    Since $p>2$,
    \begin{align*}
        \left(1+\frac{\mu}{\lambda}(b/\sigma)^{p} + \frac{\kappa p}{\lambda}(b/\sigma)^{\beta}\right)^{2/p-1} \le 1.
    \end{align*}
    By $\Lambda_3 > (p-1)\mu_s$,
    \[
        \frac{p-1}{\Lambda_3/\mu_s} + \frac{2\mu_s m}{p\Lambda_3} + C_4 (\epsilon_1^{\beta/2}+\epsilon_1^{\frac{\beta-1}{2}} + \epsilon_1^{p/2-1})< 1
    \]
    as long as $m,\epsilon_1$ is sufficiently small. By Proposition \ref{prop g(0)}, $\beta \ge 2$. If $\beta>2$, then
    \[ 
        (\mu_s^{-1}\epsilon_1)^{\beta/2-1}\frac{\sigma}{(p-1)\lambda}\left(\beta(\beta-1) + 2m\right) + C_4 (\epsilon_1^{\beta/2}+\epsilon_1^{\frac{\beta-1}{2}} + \epsilon_1^{p/2-1}) \lesssim \epsilon_1^{\beta/2-1}< 1
    \]
    provided that $m,\epsilon_1$ is sufficiently small. If $\beta=2$, then $\kappa\le \frac{\lambda}{2}$ and so
    $$
        (\mu_s^{-1}\epsilon_1)^{\beta/2-1}\frac{\sigma}{(p-1)\lambda}\left(\beta(\beta-1) + 2m\right) + C_4 (\epsilon_1^{\beta/2}+\epsilon_1^{\frac{\beta-1}{2}} + \epsilon_1^{p/2-1})
    $$
    $$
        \le \frac{m+1}{p-1} + C \epsilon_1^{\frac{\beta-1}{2}}< 1
    $$
    given that $m,\epsilon_1$ is sufficiently small. In conclusion, we can pick $m,\epsilon_1$ sufficiently small to ensure there exists $c_2\in (0,1)$, such that the right side of \eqref{est2-3} can be controlled by
    \[
        \delta(u,v) + c_2 \gamma + C_4(\gamma^{p/2} + f(\alpha,\beta,u,v)).
    \]
    As a result, by the estimate of $f(\alpha,\beta,u,v)$ and similar deduction as above, we directly obtain \eqref{est2-4}.
\end{proof}

\begin{proof}[\bf Proof of Theorem \ref{main thm}]
    If case \textbf{(I)} happens, we have $g''(\tilde{t})>0$, $\forall\, \tilde{t}\in \mathcal{S}$. Since
    \[
        |b/a-t_0| = \left|\frac{b/\sigma}{a/\sigma}-t_0\right| \lesssim \epsilon_1^{1/2},
    \]
    by $g''(t_0)>0$, for $\epsilon_1$ sufficiently small, we have
    \[
        g(b/a)-g(t_0) = \frac{1}{2}g''(t_0)(b/a-t_0)^2 + o((b/a-t_0)^2) \ge \frac 14 g''(t_0)(b/a-t_0)^2,
    \]
    and so Lemma \ref{Lemma_t>0} (if $t_0>0$) or Lemma \ref{Lemma_t=0} (if $t_0=0$) yields
    \[
        \gamma \lesssim \delta(u,v),\quad (b/a-t_0)^2 \lesssim g(b/a)-g(t_0) \lesssim \frac{\delta(u,v)}{\sigma^2}.
    \]
    Hence,
    \begin{align*}
        \|\nabla(u-\sigma w)\|_2^2 + \|\nabla(v-t_0\sigma w)\|_2^2={}& \gamma + (a-\sigma)^2\mu_s + (b-t_0\sigma)^2\mu_s \\
        ={}& \gamma + \frac{a^2\mu_s}{1+t_0^2}(b/a-t_0)^2 \\
        \lesssim{}& \delta(u,v).
    \end{align*}

    If case \textbf{(II)} happens, by Proposition \ref{prop:finite-minimizer}, we see that $\tilde{\mathcal{S}}$ is a singleton, which equals to one of $\{0\}$, $\{\sqrt{\frac{2-\beta}{2-\alpha}}\}$, or $\{\infty\}$.
    By symmetry, we may assume that $\tilde{\mathcal{S}}=\{\sqrt{\frac{2-\beta}{2-\alpha}}\}$ or $ \tilde{\mathcal{S}}=\{0\}$. 
    
    If $\tilde{\mathcal{S}}=\{\sqrt{\frac{2-\beta}{2-\alpha}}\}$ occurs, let $t_0=\sqrt{\frac{2-\beta}{2-\alpha}}$, by  Proposition \ref{prop g(t)}, $g^{(i)}(t_0)=0$ for $i=1,2,3$ and $g^{(4)}(t_0)>0$. Hence, for $\epsilon_1$ sufficiently small, we have
    \[
        g(b/a)-g(t_0) = \frac{1}{24}g^{(4)}(t_0)(b/a-t_0)^4 + o((b/a-t_0)^4) \ge \frac 1{30} g^{(4)}(t_0)(b/a-t_0)^4.
    \]
    Thus when $\epsilon_1$ is sufficiently small, Lemma \ref{Lemma_t>0} yields
    \[
        \gamma \lesssim \delta(u,v),\quad (b/a-t_0)^4 \lesssim g(b/a)-g(t_0)\lesssim \frac{\delta(u,v)}{\sigma^2}.
    \]
    Hence,
    \begin{align*}
        \|\nabla(u-\sigma w)\|_2^2 + \|\nabla(v-t_0\sigma w)\|_2^2 ={}& \gamma + \frac{a^2\mu_s}{1+t_0^2}(b/a-t_0)^2\\
        \lesssim{}& \delta(u,v) + \sigma^2\sqrt{\frac{\delta(u,v)}{\sigma^2}} \\
        \lesssim{}& \delta(u,v)^{1/2}(\|\nabla u\|_2^2 +\|\nabla v\|_2^2)^{1/2}.
    \end{align*}
    Here we use the fact that
    \begin{align*}
        \sigma = \frac{\int \nabla(u+t_0v)\cdot \nabla w\,\mathrm{d}x}{(1+t_0^2)\mu_s} \lesssim \|\nabla(u+t_0v)\|_2 \lesssim (\|\nabla u\|_2^2+\|\nabla v\|_2^2)^{1/2}.
    \end{align*}

    If $\tilde{\mathcal{S}}=\{0\}$ occurs,  by $g''(0)=0$, Proposition \ref{prop g(0)} gives that $g^{(3)}(0)=0$ and $g^{(4)}(0)>0$. Then for $\epsilon_1$ sufficiently small,
    \[
        g(b/\sigma)-g(0) =\frac 1{24}g^{(4)}(0)(b/\sigma)^4 + o((b/\sigma)^4)\ge \frac 1{30}(b/\sigma)^4.
    \]
    Thus, Lemma \ref{Lemma_t=0} yields
    \[ 
        \gamma\lesssim \delta(u,v),\quad (b/\sigma)^4\lesssim g(b/\sigma)-g(0) \lesssim \frac{\delta(u,v)}{\sigma^2}.
    \]
    Hence,
    \begin{align*}
        & \|\nabla(u-\sigma w)\|_2^2 + \|\nabla v\|_2^2 = \gamma + b^2\mu_s\\ & \lesssim \delta(u,v) + \delta(u,v)^{1/2}k \lesssim \delta(u,v)(\|\nabla u\|_2^2 + \|\nabla v\|_2^2)^{1/2}.
    \end{align*}
    Here we use the fact that
    \[
        \sigma=\frac{\int \nabla u\cdot\nabla w \,\mathrm{d}x}{\mu_s} \lesssim \|\nabla u\|_2 \lesssim (\|\nabla u\|_2^2 + \|\nabla v\|_2^2)^{1/2}.
    \]
    We complete the proof.
\end{proof}

\begin{remark}\label{rmk11}
    The exponents $1/2$ and $1$ in Theorem \ref{main thm} are sharp.
    
    In fact, when $\tilde{\mathcal{S}} \neq \emptyset$, Proposition~\ref{prop:finite-minimizer} implies that $\tilde{\mathcal{S}}$ is a singleton, equal to one of $\{0\}$, $\{\sqrt{\frac{2-\beta}{2-\alpha}}\}$, or $\{\infty\}$.
    Due to the symmetry between the cases $\tilde{\mathcal{S}} = \{0\}$ and $\tilde{\mathcal{S}} = \{\infty\}$, we may assume without loss of generality that either $\tilde{\mathcal{S}} = \{\sqrt{\frac{2-\beta}{2-\alpha}}\}$ or $\tilde{\mathcal{S}} = \{0\}$.

    Define $t_0$ as follows:
    \[ t_0 =
        \begin{cases}
        \sqrt{\frac{2-\beta}{2-\alpha}}, & \text{if $\tilde{\mathcal{S}} = \left\{\sqrt{\frac{2-\beta}{2-\alpha}}\right\}$}, \\
        0, & \text{if $\tilde{\mathcal{S}} = \{0\}$}.
        \end{cases}
    \]
    By Propositions \ref{prop g(t)} and \ref{prop g(0)}, we obtain the expansion:
    \[
        g(t) = g(t_0) + \frac{g^{(4)}(t_0)}{24}(t - t_0)^4 + o((t - t_0)^4).
    \]
    Now, let $(u,v) = (aw, bw)$ with $(a,b) = (1, t_0 + \varepsilon)$, where $w \in \mathcal{M}'$ and $\|w\|_{(s)} = 1$. Define $\sigma = \frac{a + t_0 b}{1 + t_0^2}$. Then, for sufficiently small $\varepsilon$, we have:
    \begin{align*}
        \inf_{(U_0,V_0)\subset \mathcal{M}'} \left[\|\nabla (u - U_0)\|_2^2 + \|\nabla(v - V_0)\|_2^2\right]
        &= (a - \sigma)^2\mu_s + (b - t_0\sigma)^2\mu_s \\
        &= \frac{a^2\mu_s}{1 + t_0^2}(b/a - t_0)^2 \\
        &= O(\varepsilon^2).
    \end{align*}
    On the other hand,
    \begin{align*}
        \delta(u,v) &= (a^2 + b^2)\mu_s - g(t_0)\mu_s(\lambda a^p + \mu b^p + \kappa p a^\alpha b^\beta)^{2/p} \\
        &= \mu_s (g(b/a) - g(t_0))(\lambda a^p + \mu b^p + \kappa p a^\alpha b^\beta)^{2/p} \\
        &= \mu_s(\lambda a^p + \mu b^p + \kappa p a^\alpha b^\beta)^{2/p} \left( \frac{g^{(4)}(t_0)}{24}(b/a - t_0)^4 + o((b/a - t_0)^4) \right) \\
        &= O(\varepsilon^4).
    \end{align*}
    Therefore, the exponent $1/2$ is sharp in case \emph{\textbf{(II)}}.

    For case \emph{\textbf{(I)}}, i.e. $\tilde{\mathcal{S}} = \emptyset$, fix any $t_0 \in \mathcal{S}$. By symmetry, we may assume $0 \leq t_0 < \infty$.
    Then we have the expansion:
    \[
        g(t) = g(t_0) + \frac{g''(t_0)}{2}(t - t_0)^2 + o((t - t_0)^2).
    \]
    Choose $(u,v) = (aw, bw)$ with $(a,b) = (1, t_0 + \varepsilon)$. Following similar computations as before, we obtain:
    \begin{gather*}
        \delta(u,v) = O(\varepsilon^2), \\
        \inf_{(U_0,V_0)\subset\mathcal{M}} \left[\|\nabla (u - U_0)\|_2^2 + \|\nabla(v - V_0)\|_2^2\right] = O(\varepsilon^2).
    \end{gather*}
    This shows that the exponent $1$ is sharp.
\end{remark}

\begin{proof}[\bf Proof of Corollary \ref{main_cor}]
    The key idea is the transformation in \cite{Toshio1997, Lam2017}. For
    $$
        u,v\in \{u \in L^{2^*(s)}(\R^N, |x|^{-(N-s)(1- {\ell})-s}\,\mathrm{d}x):\nabla u \in L^2(\R^N,|x|^{-(N-2)(1- {\ell})}\,\mathrm{d}x)\},
    $$
    let $\tilde{u}(x)=\tilde{u}(r\theta) := {\ell}^{1/2}u(r^{1/ {\ell}}\theta)$, $\tilde{v}(x) :=  {\ell}^{1/2}v(r^{1/ {\ell}}\theta)$, then
    \begin{equation}\label{computation_nabla}
        \begin{aligned}
            & {\int_{\R^N}}  |\nabla \tilde{u}|^2\,\mathrm{d}x\\
            & =  \int_{\S^{N-1}}\int_0^\infty (|\nabla_r\tilde{u}|^2 + r^{-2}|\nabla_\theta \tilde{u}|^2)r^{N-1}\, {\mathrm{d}r\mathrm{d}\theta}\\
            & =   {\ell}\int_{\S^{N-1}}\int_0^\infty ( {\ell}^{-2}r^{2/ {\ell}-2}|\nabla_r u|^2(r^{1/ {\ell}}\theta) +  {r^{2/\ell}}|\nabla_\theta u|^2(r^{1/ {\ell}}\theta)) r^{N-1}\, {\mathrm{d}r\mathrm{d}\theta} \\
            & =  \int_{\S^{N-1}}\int_0^\infty (|\nabla_ru|^2 +  {\ell}^2 r^{-2}|\nabla_\theta u|^2)r^{(N-1) {\ell}+1- {\ell}}\, {\mathrm{d}r\mathrm{d}\theta} \\
            & \le \int_{\S^{N-1}}\int_0^\infty (|\nabla_r u|^2 + r^{-2}|\nabla_\theta u|^2)r^{(N-2)( {\ell}-1)+N-1}\, {\mathrm{d}r\mathrm{d}\theta} \quad \hbox{ {(since $\ell<1$)}}\\
            &  =  {\int_{\R^N}}  |x|^{-(N-2)(1- {\ell})}|\nabla u|^2\,\mathrm{d}x.
        \end{aligned}
    \end{equation}
    And \eqref{computation_nabla} takes  equality   if and only if $u$ is radial. Similarly,
    \begin{align*}
        {\int_{\R^N}}  \frac{1}{|x|^s}|\tilde{u}|^p\,\mathrm{d}x ={}&  {\ell}^{p/2}\int_{\S^{N-1}}\int_0^\infty |u(r^{1/ {\ell}}\theta)|^p r^{-s+N-1}\, {\mathrm{d}r\mathrm{d}\theta} \\
        ={}&  {\ell}^{p/2+1}\int_{\S^{N-1}}\int_0^\infty |u(r\theta)|^p r^{(-s+N-1) {\ell}+ {\ell}-1}\, {\mathrm{d}r\mathrm{d}\theta} \\
        ={}& {\ell}^{p/2+1} {\int_{\R^N}}  |x|^{-(N-s)(1- {\ell})-s}|u|^p\,\mathrm{d}x,\\
         {\int_{\R^N}}  \frac{1}{|x|^s} |\tilde{u}|^\alpha|\tilde{v}|^\beta ={}&  {\ell}^{p/2+1}  {\int_{\R^N}}  |x|^{-(N-s)(1- {\ell})-s}|u|^\alpha|v|^\beta\,\mathrm{d}x.
    \end{align*}
    Hence,
    \begin{equation}\label{delta_compare}
        \begin{aligned}
           & \delta(\tilde{u},\tilde{v}) \\
           & \le{}   {\int_{\R^N}}  |x|^{-(N-2)(1- {\ell})} (|\nabla u|^2 + |\nabla v|^2)\,\mathrm{d}x \\
            & \quad - S_{\alpha,\beta,\lambda,\mu}  {\ell}^{2/p+1} \left( {\int_{\R^N}}  |x|^{-(N-s)(1- {\ell})-s} (\lambda|u|^p + \mu|v|^p +\kappa p |u|^\alpha|v|^\beta) \,\mathrm{d}x \right)^{2/p} \\
            & =  \delta_{{\ell}}(u,v).
        \end{aligned}
    \end{equation}
    By $\delta(\tilde{u},\tilde{v})\ge 0$, we get \eqref{class1}. At the same time, we can derive minimizers of \eqref{class1} through minimizers of \eqref{HS-double}: if $\delta_l(u,v)=0$, then $\delta(\tilde{u},\tilde{v})=0$ and $(\tilde{u},\tilde{v})\in \mathcal{M}'$, implying $(u,v)\in \mathcal{M}_{ {\ell}}'$. On the other hand, if $(u,v)\in \mathcal{M}_{ {\ell}}'$, then they are radial and \eqref{computation_nabla} takes equality, implying $\delta_{ {\ell}}(u,v) = \delta(\tilde{u},\tilde{v})=0$. The proof give the best constant $S_{\alpha,\beta,\lambda,\mu} {\ell}^{2/p+1}$ as well.

    By \eqref{computation_nabla}, we also have
    \begin{equation}\label{computation_nabla_below}
        \begin{aligned}
            {\int_{\R^N}} |\nabla u|^2\,\mathrm{d}x ={}& \int_{\S^{N-1}}\int_0^\infty (|\nabla_ru|^2 +  {\ell}^2 r^{-2}|\nabla_\theta u|^2)r^{(N-1) {\ell}+1- {\ell}}\, {\mathrm{d}r\mathrm{d}\theta} \\
            \ge{}&  {\ell}^2 \int_{\S^{N-1}}\int_0^\infty (|\nabla_r u|^2 + r^{-2}|\nabla_\theta u|^2) r^{(N-2)( {\ell}-1)+N-1}\, {\mathrm{d}r\mathrm{d}\theta} \\
            ={}&  {\ell}^2  {\int_{\R^N}} |x|^{-(N-2)( {\ell}-1)}|\nabla u|^2\,\mathrm{d}x,
        \end{aligned}
    \end{equation}
    and so for Case \textbf{(I)}, by Theorem \ref{main thm} and \eqref{delta_compare}-\eqref{computation_nabla_below},
    \begin{align*}
        & \inf_{(U_0,V_0)\in \mathcal{M}_{ {\ell}}'} {\int_{\R^N}} |x|^{-(N-2)( {\ell}-1)}(|\nabla(u-U_0)|^2+|\nabla(v-V_0)|^2)\,\mathrm{d}x \\
        &\le  {\ell}^{-2} \inf_{(\tilde{U}_0,\tilde{V}_0)\in\mathcal{M}}  {\int_{\R^N}} (|\nabla(\tilde{u}-\tilde{U}_0)|^2 + |\nabla(\tilde{v}-\tilde{V}_0)|^2)\,\mathrm{d}x \\
        &\le  {\ell}^{-2}C\delta(\tilde{u},\tilde{v}) \\
        &\le  {\ell}^{-2}C\delta_{ {\ell}}(u,v).
    \end{align*}
    For Case \textbf{(II)}, by Theorem \ref{main thm} and \eqref{computation_nabla}-\eqref{computation_nabla_below},
    \begin{align*}
        & \inf_{(U_0,V_0)\in \mathcal{M}'_{ {\ell}}} \frac{\int_{\R^N} |x|^{-(N-2)(1- {\ell})}(|\nabla(u-U_0)|^2 + |\nabla(v-V_0)|^2)\,\mathrm{d}x}{\int_{\R^N} |x|^{-(N-2)( {\ell}-1)}(|\nabla u|^2+|\nabla v|^2)\,\mathrm{d}x}\\
        &\le  {\ell}^{-2} \inf_{(\tilde{U}_0,\tilde{V}_0)\in\mathcal{M}'} \frac{\|\nabla(\tilde{u}-\tilde{U}_0)\|_2^2 + \|\nabla(\tilde{v}-\tilde{V}_0)\|_2^2}{\int_{\R^N} |x|^{-(N-2)( {\ell}-1)}(|\nabla u|^2+|\nabla v|^2)\,\mathrm{d}x} \\
        &\le  {\ell}^{-2}C \frac{\delta(\tilde{u},\tilde{v})^{1/2}(\|\nabla \tilde{u}\|_2^2 + \|\nabla\tilde{v}\|_2^2)^{1/2}}{\int_{\R^N} |x|^{-(N-2)( {\ell}-1)}(|\nabla u|^2+|\nabla v|^2)\,\mathrm{d}x} \\
        &\le  {\ell}^{-2}C \left(\frac{\delta_{ {\ell}}(u,v)}{\int_{\R^N} |x|^{-(N-2)( {\ell}-1)}(|\nabla u|^2+|\nabla v|^2)\,\mathrm{d}x}\right)^{1/2}. \qedhere
    \end{align*}
\end{proof}

\vskip0.2in

\begin{center}{DECLARATIONS}\end{center}

\noindent {\bf Data availability statement.} Data sharing not applicable to this article as no data sets were generated
or analysed during the current study.

\noindent{\bf Competing interests. } The authors have no relevant financial or non-financial interests to disclose.

\noindent{\bf Authors contributions. } All authors contributed equally to the writing of this paper. All authors
read and approved the final manuscript.

\end{document}